\documentclass[reqno, 12pt]{amsart}

\usepackage{amsfonts, amsthm, amsmath, amssymb}
\usepackage[hidelinks]{hyperref}
\hypersetup{
  pdftitle={On the average of class numbers of indefinite binary quadratic forms},
  pdfauthor={Yijie Diao},
  pdfsubject={Average class numbers of indefinite binary quadratic forms in discriminant order},
  pdfkeywords={Indefinite binary quadratic form, class number, real quadratic field, Pell equation, fundamental unit, determinant method}
}
\usepackage{microtype,needspace}
\usepackage[margin=1.5in]{geometry}

\numberwithin{equation}{section}

\newtheorem{theorem}{Theorem}[section]
\newtheorem{lemma}[theorem]{Lemma}
\newtheorem{proposition}[theorem]{Proposition}
\newtheorem{corollary}[theorem]{Corollary}

\renewcommand{\leq}{\leqslant}

\renewcommand{\geq}{\geqslant}

\renewcommand{\bar}{\overline}

\newcommand{\ZZ}{\mathbb{Z}}
\newcommand{\NN}{\mathbb{N}}
\newcommand{\QQ}{\mathbb{Q}}

\newcommand{\leg}[2]{\left(\frac{#1}{#2}\right)}

\DeclareMathOperator{\disc}{disc}
\DeclareMathOperator{\Nm}{N}

\begin{document}

\date{\today}
\title[Class numbers of indefinite binary quadratic forms]{On the average of class numbers of indefinite binary quadratic forms}
\author{Yijie Diao}
\address{ISTA\\
Am Campus 1\\
3400 Klosterneuburg\\
Austria}
\email{yijie.diao@ist.ac.at}
\subjclass[2020]{11E41 (11D09, 11R11, 11R29, 11N36)}

\begin{abstract}
We prove that the sum of class numbers of primitive indefinite binary quadratic forms with integer coefficients over positive nonsquare discriminants less than $X$ is $O(X^{3/2}\sqrt{\log\log X}/(\log X)^{3/2})$.
This improves the longstanding classical upper bound $O(X^{3/2}/\log X)$ and implies the same upper bound for the corresponding sum over real quadratic fields.
Our main input is a uniform upper bound for the number of small fundamental solutions to Pell's equation.
In particular, we prove that, for almost all nonsquare positive integers $D$, the fundamental solution $\varepsilon_D$ of $t^2-Du^2=1$ is larger than any fixed power of $D$.
\end{abstract}

\maketitle
\raggedbottom
\thispagestyle{empty}
\enlargethispage{2pt}
\setcounter{tocdepth}{1}
\tableofcontents

\section{Introduction}

\subsection{Average of class numbers}

Let $\mathcal D$ be the set of positive nonsquare integers congruent to $0$ or $1$ modulo $4$.
For $D\in\mathcal D$, write $h^+(D)$ for the number of $\operatorname{SL}_2(\ZZ)$-classes of primitive indefinite binary quadratic forms $Ax^2+Bxy+Cy^2$ with integer coefficients and discriminant $B^2-4AC=D$.
We study the sum
\[
  S^+(X) = \sum_{\substack{D\in\mathcal D\\D<X}}h^+(D).
\]

Let $K$ be a real quadratic field with ring of integers $\mathcal O_K$, discriminant $D_K$, and ideal class number $h(K)=\#\operatorname{Cl}(K)$.
Put
\[
 S(X)=\sum_{\substack{K/\QQ\ \mathrm{quadratic}\\0<D_K<X}}h(K).
\]
The correspondence between forms and ideals identifies $h^+(D_K)$ with the narrow class number of $K$ (see Lamzouri~\cite[p.~566]{lamzouri}), so
\[
 h(K)\leq h^+(D_K)\leq2h(K).
\]
Thus
\begin{equation}\label{e:field-form-average-comparison}
 S(X)
 \leq
 \sum_{\substack{D\in\mathcal D,\ D<X\\D\ \mathrm{fundamental}}}h^+(D)
 \leq
 S^+(X).
\end{equation}

Let $\eta_D^+>1$ be the least totally positive unit of the quadratic order $\mathcal O_D=\ZZ[(D+\sqrt D)/2]$, given by
\begin{equation}\label{e:totally-positive-unit}
 \eta_D^+=\min_{\substack{t,u\in\ZZ_{>0}\\t^2-Du^2=4}}\frac{t+u\sqrt D}2.
\end{equation}
Write $\chi_D(n)$ for the quadratic Dirichlet character defined by the Kronecker symbol.
Dirichlet's class number formula~\cite[(1.1)]{lamzouri} gives
\begin{equation}\label{e:ordinary-class-formula}
 h^+(D)\log\eta_D^+=\sqrt D L(1,\chi_D).
\end{equation}
Since $\eta_D^+>\sqrt D$, we have the elementary lower bound
\begin{equation}\label{e:elementary-unit-lower-bounds}
 \log\eta_D^+\gg\log D.
\end{equation}

In his 1984 paper, Hooley studied class number averages through the distribution of fundamental solutions of Pell's equation.
The analogue\footnote{Hooley used the \textit{determinant} $d=B^2-AC$ of quadratic forms $Ax^2+2Bxy+Cy^2$, so his class number is $h^+(4d)$ in our notation.
} of his conjecture~\cite[Conjecture~7]{hooley} in our setting is that, for some constant $C>0$,
\[
 S^+(X)\sim CX(\log X)^2.
\]
However, we are still very far from this conjecture.
Sarnak~\cite[p.~231, (79)]{sarnak-reciprocal} recorded the following classical upper bound
\begin{equation}\label{e:classical-order-average}
 S^+(X)
 \ll\frac{X^{3/2}}{\log X}.
\end{equation}
As Sarnak observed, this follows directly from~\eqref{e:elementary-unit-lower-bounds} and the asymptotic
\[
 \sum_{\substack{D\in\mathcal D\\D<X}}h^+(D)\log\eta_D^+
 =\frac{\pi^2}{18\zeta(3)}X^{3/2}+O(X\log X),
\]
predicted by Gauss and proved in this form by Siegel\footnote{~The upper bound~\eqref{e:classical-order-average} also has an elementary proof.
Indeed, partial summation and elementary character sum estimates give
\[
 \sum_{\substack{D\in\mathcal D\\D<X}}L(1,\chi_D)\ll X.
\]
Combining this with~\eqref{e:elementary-unit-lower-bounds} in the class number formula~\eqref{e:ordinary-class-formula} gives~\eqref{e:classical-order-average}.}~\cite{siegel-average} in 1944.
The product $h^+(D)\log\eta_D^+$ is therefore well understood on average, but the two factors separately are much harder to control.
Fouvry and Jouve~\cite[Theorem~3 and (56)]{fj-positive} obtained an explicit upper bound for the leading constant in the analogous estimate for the sum of $h(K)$ over real quadratic fields with $4\parallel D_K$, while the order of magnitude of the classical upper bound~\eqref{e:classical-order-average} has not been improved for several decades.

We give the first improvement to the classical upper bound~\eqref{e:classical-order-average} by a factor $\sqrt{\log X/\log\log X}$.

\begin{theorem}\label{T:field-class-average}
For $X\geq3$, we have
\[
 \sum_{\substack{K/\QQ\ \mathrm{quadratic}\\0<D_K<X}}h(K)
 \leq\sum_{\substack{D\in\mathcal D\\D<X}}h^+(D)
 \ll \frac{X^{3/2}\sqrt{\log\log X}}{(\log X)^{3/2}}.
\]
\end{theorem}

Genus theory gives a lower bound of order $X\log X$ for the sum of $h^+(D)$ over positive fundamental discriminants $D<X$.
Fouvry and Kl\"uners~\cite[Theorem~3]{fk-weighted} improved the leading constant in this lower bound for the subfamily $D\equiv1\bmod4$.

By ordering quadratic forms by their fundamental units, Sarnak~\cite[Theorem~2.1]{sarnak} proved
\[
 \sum_{\substack{D\in\mathcal D\\\eta_D^+\leq Y}}h^+(D)
 \sim\operatorname{Li}(Y^2),
 \quad \text{and} \quad
 \#\{D\in\mathcal D:\eta_D^+\leq Y\}
 \sim\frac{35}{16}Y.
\]
This class number sum counts oriented primitive closed geodesics on the modular surface of length at most $2\log Y$, so the prime geodesic theorem gives the main term $\operatorname{Li}(Y^2)$.
Siegel averaged the product $h^+(D)\log\eta_D^+$, while Sarnak ordered the class numbers by $\eta_D^+$.

For primitive indefinite ternary quadratic forms, Gamburd, Ghosh, Sarnak, and Whang~\cite[(B)]{gamburd-ghosh-sarnak-whang} proved a class number asymptotic of order $X\log X$ for positive determinants less than $X$.
For ternary forms, genera containing more than one class contribute only a term of lower order, but for binary forms the number of classes in each genus remains difficult to control (see~\cite[Remark~1.5]{gamburd-ghosh-sarnak-whang}).

\subsection{Torsion in class groups}

For a real quadratic field $K$ and an integer $n\geq2$, let $h_n(K)=\#\operatorname{Cl}(K)[n]$, and put
\[
 S_n(X)=\sum_{0<D_K<X}h_n(K).
\]
The Cohen--Lenstra heuristics~\cite{cohen-lenstra} and Gerth's extension~\cite{gerth} predict, for suitable constants $c_n>0$,
\[
 S_n(X)\sim
 \begin{cases}
  c_nX,&n\text{ odd},\\
  c_nX\log X,&n\text{ even}.
 \end{cases}
\]
For $n=2$, it follows from Gauss's genus theory.
The asymptotic is also known for $n=3,4,6$ (see~\cite{chan-koymans-pagano-sofos,davenport-heilbronn,fk-weighted,koymans-lemke-oliver-sofos-thorne}).
Smith~\cite{smith-i,smith-ii} established the predicted distributions of higher $2$-power class ranks.

For general $n$, the known upper bounds for $S_n(X)$ give only a small saving in the exponent $3/2$ that follows from~\eqref{e:field-form-average-comparison} and~\eqref{e:classical-order-average}.
For $n\geq2$, Koymans and Thorner~\cite[Theorem~1.1 and (1.14)]{koymans-thorner} proved
\[
 S_n(X)\ll_{n,\varepsilon}X^{3/2-1/(n+1)+\varepsilon}.
\]
For primes $\ell\geq5$, Chan and Koymans~\cite[Theorem~1.2]{chan-koymans} obtained
\[
 S_\ell(X)\ll_{\ell,\varepsilon}X^{3/2-3/(2\ell+2)+\varepsilon}.
\]

Let $h_{\mathrm{odd}}(K)$ denote the number of elements of odd order in $\operatorname{Cl}(K)$.
Browning and Chan~\cite[p.~1679]{browning-chan} proved
\[
 \sum_{\substack{K/\QQ\ \mathrm{quadratic}\\0<D_K<X}}h_{\mathrm{odd}}(K)
 \ll\frac{X^{3/2}}{(\log X)^{3/2}}.
\]
They combined genus theory with a mean value estimate of Wilson~\cite[Corollary~1.6]{wilson} carrying the weight $2^{-\omega(D_K)}$.
We can also improve this bound by a factor $\sqrt{\log X/\log\log X}$.

\begin{theorem}\label{T:odd-class-average}
For $X\geq3$, we have
\[
 \sum_{\substack{K/\QQ\ \mathrm{quadratic}\\0<D_K<X}}h_{\mathrm{odd}}(K)
 \ll\frac{X^{3/2}\sqrt{\log\log X}}{(\log X)^2}.
\]
\end{theorem}

\subsection{Size of fundamental solutions to Pell equations}

To bound the average class number using~\eqref{e:ordinary-class-formula}, we need to control how often $\log\eta_D^+$ is small.
We do this by studying the size of fundamental solutions of Pell's equation.
Hooley~\cite[Section~7]{hooley} used heuristics for the distribution of these solutions to formulate his conjecture on average class numbers.

For every positive nonsquare integer $D$, define the fundamental solution of Pell's equation by
\begin{equation}\label{e:pell-fundamental-unit}
 \varepsilon_D=\min_{\substack{t,u\in\ZZ_{>0}\\t^2-Du^2=1}}(t+u\sqrt D).
\end{equation}
Comparing~\eqref{e:pell-fundamental-unit} with~\eqref{e:totally-positive-unit} gives $\varepsilon_D=\eta_{4D}^+$.
For $D\in\mathcal D$, Lemma~\ref{L:unit-comparison} gives
\[
 \eta_D^+\leq\varepsilon_D\leq(\eta_D^+)^3.
\]
Thus lower bounds for $\log\varepsilon_D$ yield corresponding lower bounds for $\log\eta_D^+$, with only an absolute constant loss.
The classical Pell formulation~\eqref{e:pell-fundamental-unit} gives the factorization $(t-1)(t+1)=Du^2$ used in our counting argument.

For $\theta\geq0$, put
\begin{equation}\label{e:exceptional-count}
 E_\theta(X)=
 \#\{0<D<X:D\text{ is nonsquare and }\varepsilon_D\leq D^\theta\}.
\end{equation}
The elementary inequality $\varepsilon_D>\sqrt D$ gives $E_\theta(X)=0$ for $\theta\leq1/2$.
For $1/2<\theta\leq1$, Hooley~\cite[Theorem~1]{hooley} proved an asymptotic formula showing that
\[
 E_\theta(X) \sim c_\theta X^{1/2}(\log X)^2,
\]
for some explicit constant $c_\theta>0$.
In the same paper, Hooley~\cite[Conjecture~1]{hooley} conjectured an asymptotic formula of the same order for every fixed $\theta>1/2$.

The full conjecture remains out of reach with current techniques.
A simpler problem is to determine for which $\theta$ one can prove $E_\theta(X)=o(X)$ as $X\to\infty$.
Hooley~\cite[Corollary, p.~104]{hooley} showed that, for every $\theta<3/2$, the fundamental solution satisfies $\varepsilon_D>D^\theta$ for almost all positive nonsquare integers $D$.
Fouvry and Jouve~\cite[Corollary~1]{fj} studied this problem and extended the range to $\theta<7/4$.
Reuss~\cite[Corollary~10]{reuss} further extended it to any $\theta<3$; all references to his work use the numbering in his thesis.
The following theorem proves a power saving bound for $E_\theta(X)$ for every $\theta\geq3$, uniformly in a range where $\theta$ grows with $X$.

\begin{theorem}\label{T:uniform-exceptional}
There is an absolute constant $c>0$ such that, for all sufficiently large $X$,
\[
 E_\theta(X)\ll X^{1-c/\theta^2},
 \qquad
 \text{uniformly for }3\leq\theta\leq c\sqrt{\frac{\log X}{\log\log X}}.
\]
The implied constant is absolute.
\end{theorem}

We do not attempt to optimize $c$, since a better value would only improve the implied constants in Theorems~\ref{T:field-class-average} and~\ref{T:odd-class-average}.

The following result is a direct corollary of Theorem~\ref{T:uniform-exceptional}.
\begin{corollary}\label{C:pell-unit-growth}
As $X\to\infty$, almost all  nonsquare integers $0<D<X$ satisfy
\[
 \log\varepsilon_D
 \gg\frac{(\log D)^{3/2}}{\sqrt{\log\log D}}.
\]
In particular, for any fixed $\theta>0$, almost all such $D$ satisfy $\varepsilon_D>D^\theta$.
\end{corollary}

For $1<\theta\leq3/2$, Fouvry~\cite[Theorem~1]{fouvry} proved the lower bound
\[
 E_\theta(X)\geq\bigl(b_\theta-o(1)\bigr)X^{1/2}(\log X)^2,
\]
for an explicit constant $b_\theta>0$ depending on $\theta$.
Bourgain~\cite{bourgain} and Xi~\cite{xi} further improved the leading constant within the range $1<\theta\leq3/2$.
Fouvry and Jouve~\cite{fj-prescribed} also proved lower bounds of this order with $\varepsilon_D$ restricted to prescribed intervals.

\subsection*{Notation}

The letter $\varepsilon$ denotes an arbitrarily small positive number, which may change at each occurrence.
The letter $C$ denotes a positive constant that may have a different value at each occurrence.
When specifying dyadic ranges, we write $u\sim U$ to mean $u\in[U,2U)$.
For a nonzero integer $n$, let $\tau(n)$ denote the number of its positive divisors.
An integer pair $(x_1,x_2)$ is primitive if $\gcd(x_1,x_2)=1$.
For vectors, $\lVert\cdot\rVert$ denotes the Euclidean norm.
For a matrix $T$, we put $\lVert T\rVert=\sup_{\lVert\mathbf v\rVert=1}\lVert T\mathbf v\rVert$.
For a polynomial $F$, $\lVert F\rVert$ is the largest absolute value of its coefficients.

\subsection*{Supplementary Lean code}

The Lean code accompanying this paper contains a formalization of Theorem~\ref{T:field-class-average}, Theorem~\ref{T:odd-class-average}, and Theorem~\ref{T:uniform-exceptional} based only on Mathlib.
It is available at the following repository:
\begin{center}
\url{https://github.com/s6yidiao-cloud/pell-lean}
\end{center}

\subsection*{Use of artificial intelligence}

The paper was originally written solely by the author.
The results were subsequently formalized in Lean using GPT-5.6 Sol.
The formalization led to a simpler proof strategy for Theorem~\ref{T:uniform-exceptional}: fixing one auxiliary degree at $1$ and varying the other, instead of optimizing both.
This allows us to count points on the auxiliary curves using the square sieve in place of the original approach inspired by Reuss~\cite{reuss} (see Section~\ref{s:main-counting-argument}).
GPT-5.6 Sol and GPT-6.1 Sol were also used to edit the paper.

\subsection*{Acknowledgements}

I would like to thank my advisor, Tim Browning, for introducing me to this problem in the first year of my PhD in 2023, and for his support and many helpful discussions over the years.

\section{Overview of the proof}\label{s:proof-overview}

We first prove Theorem~\ref{T:uniform-exceptional}.
In Section~\ref{s:main-proofs}, we use the class number formula to deduce Theorems~\ref{T:field-class-average} and~\ref{T:odd-class-average}.

\subsection{The dyadic counting function}\label{s:pell-factorizations}

We begin with squarefree $w>1$ and write $\varepsilon_w=t_w+u_w\sqrt w$.
Pell's equation gives $(t_w-1)(t_w+1)=wu_w^2$.
The factors $t_w-1$ and $t_w+1$ have greatest common divisor $1$ or $2$.
After dividing by it, we write each quotient as a square times a squarefree integer.
When $t_w$ is odd, this gives
\[
 \frac{t_w+1}{2}=e^2v,\qquad
 \frac{t_w-1}{2}=d^2u,\qquad
 uv=w,\qquad e^2v-d^2u=1,
\]
with $u_w=2de$.
When $t_w$ is even, both $w$ and $u_w$ are odd, and we have
\[
 t_w+1=e^2v,\qquad
 t_w-1=d^2u,\qquad
 uv=w,\qquad e^2v-d^2u=2,
\]
with $u_w=de$.
In either case, $u$ and $v$ are coprime and squarefree.
Interchanging the two factors when necessary gives
\begin{equation}\label{e:factor-equation}
 e^2v-d^2u=h,\qquad h\in\{\pm1,\pm2\}.
\end{equation}
Since $\gcd(d,e)^2$ divides $h$, we also have $\gcd(d,e)=1$.

Each small fundamental solution yields such a tuple $(d,e,u,v)$ of bounded size, so an upper bound for these tuples bounds the exceptional set.
As in Reuss~\cite[Section~2.6]{reuss}, we divide the variables into dyadic ranges.
For $D,E,U,V\geq1$ and $h\in\{\pm1,\pm2\}$, let $N_h(D,E,U,V)$ count the positive integral solutions of~\eqref{e:factor-equation} with $u,v$ coprime and squarefree and
\begin{equation}\label{e:scales}
 u\sim U,\quad v\sim V,\quad
 d\sim D,\quad e\sim E.
\end{equation}
If the box contains a solution, then $D^2U\asymp E^2V$ by~\eqref{e:factor-equation}, and $UV\leq uv<4UV$.
We put
\begin{equation}\label{e:dyadic-size}
 Y=DE\sqrt{UV} \asymp D^2 U \asymp E^2 V.
\end{equation}
For tuples arising from $\varepsilon_w$, the factorizations above give $uv=w$ and $u_w=de$ or $2de$.
Since $\varepsilon_w=t_w+u_w\sqrt w$ and $t_w=\sqrt{1+wu_w^2}$, this gives $Y\asymp\varepsilon_w$, so the bound on $\varepsilon_w$ translates to a bound on $Y$.
Summing over the four values of $h$, we write
\begin{equation}\label{e:dyadic-count}
 N(D,E,U,V)=\sum_{h\in\{\pm1,\pm2\}}N_h(D,E,U,V).
\end{equation}

To bound $E_\theta(X)$ from~\eqref{e:exceptional-count}, we first count squarefree $w$ with $\varepsilon_w\leq w^\theta$.
The corresponding boxes satisfy $Y\leq(4UV)^\theta$.
We pass to general nonsquare integers in Section~\ref{s:uniform-exceptional}, using the squarefree estimate with parameter $2\theta$.
Since $w=uv\asymp UV$, our aim is a bound of the form
\[
 N(D,E,U,V)\ll(UV)^{1-\delta(\theta)},
\]
uniformly for as large a range of $\theta$ as possible.

\subsection{The main counting argument}\label{s:main-counting-argument}

The determinant method is a general tool for counting integral and rational points of bounded height on algebraic varieties.
It constructs auxiliary polynomial equations satisfied by the points, reducing the count to varieties of smaller dimension.
The method originated in the work of Bombieri and Pila~\cite{bombieri-pila} and was developed further by Heath-Brown~\cite{hb-determinant} and Salberger~\cite{salberger-determinant}, among others.
For an introduction to the Bombieri--Pila determinant method, see Bloom and Lichtman~\cite{bloom-lichtman}.

Heath-Brown introduced the approximate determinant method in his work on sums and differences of three $k$th powers~\cite{hb-three-powers} and later applied it to squarefree values of $n^2+1$~\cite{hb-squarefree}.
This variant counts rational points lying close to an algebraic curve.
Reuss~\cite[Chapter~2]{reuss} extended this approach to pairs of $k$-free numbers and small solutions of Pell's equation.
For a recent development concerning rational points close to nonsingular plane curves of higher degree, see Adiceam and Marmon~\cite{adiceam-marmon}.

By~\eqref{e:factor-equation}, the ratios $s=d/e$ and $t=v/u$ satisfy
\[
 t-s^2=\frac{h}{e^2u},
\]
so the points lie close to a parabola after normalization.
For a parameter $M\geq1$, we divide the range of $s$ into $O(M)$ intervals of length $D/(EM)$ and, on each interval, form a matrix by evaluating the monomials
\[
 s^at^b,\qquad 0\leq a\leq A,\quad 0\leq b\leq B.
\]
Expanding about the parabola bounds the determinants formed from these monomials in terms of the interval length and the error $t-s^2$.
After clearing denominators, these determinants are integers, so a bound smaller than one forces them to vanish.
The columns then satisfy a nontrivial linear relation, giving an auxiliary equation
\begin{equation}\label{e:overview-auxiliary-equation}
 \Psi(d,e;u,v)=0
\end{equation}
for all points in the interval, where $\Psi$ is homogeneous of degree $A$ in $(d,e)$ and of degree $B$ in $(u,v)$.
In the ratios $(s,t)$, this equation defines an auxiliary curve.

We improve Reuss's bound for these determinants by using relations among the Taylor coefficients to show that additional terms in the determinant expansion vanish.
Section~\ref{s:determinant-example} illustrates this improvement in the case $A=1$, $B=2$, before we prove the general estimate.
This extra cancellation allows us to take fewer intervals, with $M<UV$, even when $Y$ is too large for Reuss's bound~\eqref{e:reuss-box-estimate} below to give a power saving.

Reuss~\cite[Section~2.2.1]{reuss} varied both auxiliary degrees $A$ and $B$, choosing $A/B$ close to $\log(UV)/\log(DE)$.
We fix $A=1$ and vary $B$.
This keeps the dependence on the auxiliary degree explicit and allows us to eliminate $(d,e)$.
First set aside the factors of the auxiliary form that depend only on $(u,v)$.
The remaining equation has the form
\[
 dF_1(u,v)+eF_2(u,v)=0,
\]
where $F_1,F_2$ are coprime integral homogeneous forms of the same degree $b\leq B$.
Since $\gcd(d,e)=1$, we have $(F_2(u,v),-F_1(u,v))=g(d,e)$ for some integer $g$.
The coprimality of $F_1,F_2$ implies that they cannot both vanish at $(u:v)$, so $g\ne0$.
Substituting into $e^2v-d^2u=h$ gives
\[
 h\bigl(vF_1(u,v)^2-uF_2(u,v)^2\bigr)
 =(hg)^2\ne0.
\]
Thus we count pairs $(u,v)$ for which a binary form of odd degree takes a nonzero square value.
Each $(u,v)$ gives at most two primitive $(d,e)$, as shown in Lemma~\ref{L:linear-elimination}.
We use Heath-Brown's square sieve~\cite{hb-square-sieve}, in the form of Corollary~\ref{C:anisotropic-square}, to count these square values in reduced lattice coordinates.
The resulting bound has polynomial dependence on $\theta$.

\subsection{Balanced and unbalanced ranges}\label{s:dyadic-bounds}

We first recall Reuss's bound for $N(D,E,U,V)$.
Suppose that $D,E>1$ and
\[
 \log(DE)\asymp\log(UV),
 \qquad DE\gg (UV)^{1/2},
\]
with fixed implied constants.
Reuss used a parameter $M$ defined by
\[
 \log M=\frac98\frac{\log(DE)\log(UV)}{\log Y}.
\]
For any $\varepsilon>0$ and sufficiently large $Y$, Reuss~\cite[Theorem~5]{reuss} proved
\begin{equation}\label{e:reuss-box-estimate}
 N(D,E,U,V)
 \ll_{\varepsilon}Y^\varepsilon
 \min((DEM)^{1/2}+D+E,\ (UVM)^{1/2}+U+V).
\end{equation}
The terms $U+V$ also occur in our estimate in Proposition~\ref{P:balanced-box} and prevent the required saving when one squarefree factor is too small.
We treat the following two ranges separately, first describing the bounds for fixed $\theta\geq3$.

In the balanced ranges, namely when
\[
 \min(u,v)>(UV)^{1/(100\theta^2)},
\]
Reuss's estimate~\eqref{e:reuss-box-estimate} gives no power saving once $Y\geq(UV)^{9/2}$.
For $Y=DE\sqrt{UV}\leq(4UV)^\theta$, Proposition~\ref{P:balanced-ranges} implies, under its stated hypotheses,
\[
 N(D,E,U,V)
 \ll_{\theta,\varepsilon}(UV)^{1-1/(200\theta^2)+\varepsilon}.
\]

In the unbalanced ranges, namely when
\[
 \min(u,v)\leq(UV)^{1/(100\theta^2)},
\]
Reuss~\cite[Section~2.6, p.~61]{reuss} used congruence estimates, but his bound gives no power saving for $Y\geq(UV)^3$.
We fix the smaller factor $v$, interchanging the factors if necessary.
Writing $n=ev$ and $a=-hv$, equation~\eqref{e:factor-equation} becomes
\begin{equation}\label{e:unbalanced-equation}
 d^2w=n^2+a,\qquad |a|\asymp v.
\end{equation}
For tuples arising from $\varepsilon_w$, we will show that $a$ is squarefree and $a\ne-1$ (see Section~\ref{s:unbalanced}).
We study the counting function
\begin{equation}\label{e:unbalanced-count}
 N_a(D,W)=\#\left\{(d,w,n)\in\ZZ_{>0}^3:
 \begin{array}{l}
 d^2w=n^2+a,\quad d\sim D,\\
 w\sim W,\quad w\ \text{squarefree}
 \end{array}\right\},
\end{equation}
at scales $W\asymp UV$.
When $d$ is small, we count the possible values of $n$ using the congruence $n^2\equiv-a\pmod{d^2}$.
For larger $d$, we factor $n^2+a$ in $\QQ(\sqrt{-a})$ and apply the improved approximate determinant method and square sieve in the resulting coordinates (see Sections~\ref{s:quadratic-preliminaries} and~\ref{s:unbalanced}).
Proposition~\ref{L:uniform-unbalanced-count} implies, under its stated hypotheses,
\[
 N_a(D,W)\ll_{\theta,\varepsilon}|a|^{7/8}W^{1-1/(32\theta^2)+\varepsilon}.
\]
For $w\sim UV$ with $\varepsilon_w\leq w^\theta$, summing over the values of $a$ and using Lemma~\ref{L:direct-count} for small $d$ shows that the unbalanced contribution is
\[
 \ll_{\theta,\varepsilon}(UV)^{1-1/(80\theta^2)+\varepsilon},
\]
which also gives a power saving over $UV$.

As $\theta$ grows in Theorem~\ref{T:uniform-exceptional}, a fixed $\varepsilon>0$ eventually exceeds the saving of order $\theta^{-2}$ in the exponent.
We therefore need uniform estimates with explicit logarithmic factors, which can be absorbed in this range.
These estimates are stated and proved in Sections~\ref{s:balanced} and~\ref{s:unbalanced}, as Propositions~\ref{P:balanced-ranges} and~\ref{L:uniform-unbalanced-count}, respectively.

\section{Prerequisites}\label{s:counting-tools}

This section collects estimates for counting integral points on the intersection of~\eqref{e:factor-equation} and the auxiliary equation~\eqref{e:overview-auxiliary-equation}.
The square sieve below is used in Sections~\ref{s:balanced} and~\ref{s:unbalanced} after eliminating the first coordinate pair from an auxiliary equation.
The lattice estimates allow us to sum the resulting counts over the short intervals used in the determinant method.
When an auxiliary factor depends only on the second coordinate pair, fixing that pair leaves a binary quadratic equation in the first.
The last estimate in this section gives a uniform bound for its solutions.

\subsection{A square sieve}\label{s:square-sieve}

We follow Heath-Brown~\cite[proof of Lemma~2.2]{hb-squarefull}, adapting his argument to rectangles.
Let $f\in\ZZ[X,Y]\setminus\{0\}$ be a homogeneous form with odd degree.
Let $I_1,I_2$ be real intervals with lengths
\[
 S_1=|I_1|,\qquad S_2=|I_2|,\qquad 1\leq S_2\leq S_1.
\]
Put
\[
 W=2+\max_{(x,y)\in(I_1\times I_2)\cap\ZZ^2}|f(x,y)|.
\]
The maximum is taken to be zero if the rectangle contains no integer points.
\begin{lemma}\label{L:anisotropic-square}
There is an absolute constant $C>0$ such that, if
\[
 S_2^{1/2} \geq C\log W,
\]
then
\[
 \#\{(x,y)\in (I_1\times I_2)\cap\ZZ^2:0\ne f(x,y)=\square\}
 \ll \log^2(2S_2) S_1S_2^{1/2}.
\]
The implied constant is absolute.
\end{lemma}

For $f(X,Y)=Y^3$ and $I_1=[0,S_1], I_2=[0,S_2]$, the count is $\sim S_1S_2^{1/2}$.
Thus the bound in Lemma~\ref{L:anisotropic-square} is sharp up to logarithmic factors.

\begin{proof}
Let $N$ denote the count in Lemma~\ref{L:anisotropic-square} and choose $P\geq C\log W$.
Put
\[
 \mathcal P=\{p\text{ prime}:P<p\leq4P\}.
\]
For sufficiently large $C$, we have $P\geq2$, $\#\mathcal P\asymp P/\log P$, and $(\log W)/\log P\leq\#\mathcal P/2$.
Thus each nonzero square value is divisible by at most half the primes in $\mathcal P$.
Heath-Brown's square sieve~\cite[Proof of Theorem~1]{hb-square-sieve} therefore gives
\begin{equation}\label{e:anisotropic-detector}
 N\ll\frac{S_1S_2\log P}{P}+\frac{(\log P)^2}{P^2}
 \sum_{\substack{p,q\in\mathcal P\\p\ne q}}
 \left|\sum_{(x,y)\in(I_1\times I_2)\cap\ZZ^2}\leg{f(x,y)}{pq}\right|.
\end{equation}

For odd $m\geq3$, write $e_m(t)=\exp(2\pi i t/m)$ and put
\[
 S_m(a,b)=\sum_{x,y\bmod m}\leg{f(x,y)}m e_m(ax+by).
\]
Odd homogeneity allows us to average over $(x,y)\mapsto(tx,ty)$.
For a prime $p$, the scaling identity~\cite[(2.7), p.~8]{hb-squarefull} gives
\[
 S_p(a,b)=\frac{\tau_p}{p-1}
 \sum_{x,y\bmod p}\leg{f(x,y)(ax+by)}p,
 \qquad |\tau_p|=p^{1/2}.
\]
Here $\tau_p$ is a quadratic Gauss sum.
If $(a,b)\not\equiv(0,0)\pmod p$, at most $p(p-1)$ terms are nonzero.
If $(a,b)\equiv(0,0)$, they all vanish.
For $m=pq$ with distinct $p,q\in\mathcal P$, the product formula~\cite[(2.6), p.~7]{hb-squarefull} therefore gives
\[
 S_m(0,0)=0,\qquad |S_m(a,b)|\leq m^{3/2}.
\]
These bounds require no assumptions on the factorization of $f$ modulo $p$ or $q$.

For a real interval $I$ of length at least $1$, put
\[
 \widehat I(a)=\sum_{n\in I\cap\ZZ}e_m(-an).
\]
Uniformly in $|I|$, the usual bound for geometric sums gives
\[
 \widehat I(0)=\#(I\cap\ZZ)\ll |I|,
 \qquad
 \sum_{\substack{a\bmod m\\a\ne0}}|\widehat I(a)|\ll m\log(2m).
\]
In Fourier inversion, the term with both frequencies zero vanishes because $S_m(0,0)=0$.
Separating the terms with exactly one zero frequency from those with both frequencies nonzero gives
\[
 \sum_{(x,y)\in(I_1\times I_2)\cap\ZZ^2}\leg{f(x,y)}m
 \ll m^{3/2}\log^2(2m)
       +(S_1+S_2)m^{1/2}\log(2m).
\]
Substituting this into~\eqref{e:anisotropic-detector} gives, for every $P\geq C\log W$,
\begin{equation}\label{e:square-sieve-parameter-bound}
 N\ll
 \frac{S_1S_2\log(2P)}P
 +P^3\log^2(2P)+(S_1+S_2)P\log(2P).
\end{equation}
By the hypothesis, we may take $P=S_2^{1/2}$.
Then $S_1S_2/P=S_1P$, $P^3\leq S_1P$ and $(S_1+S_2)P\leq2S_1P$.
\end{proof}

For the applications in Sections~\ref{s:balanced} and~\ref{s:unbalanced}, we use the following form, with a hypothesis on the product $S_1S_2$.
\begin{corollary}\label{C:anisotropic-square}
There is an absolute constant $C>0$ such that, if
\[
 (S_1S_2)^{1/4} \geq C\log W,
\]
then
\[
 \#\{(x,y)\in (I_1\times I_2)\cap\ZZ^2:0\ne f(x,y)=\square\}
 \ll \log^2(2S_1S_2) S_1^{5/4}S_2^{1/4}.
\]
The implied constant is absolute.
\end{corollary}

\begin{proof}
The assumption allows us to use~\eqref{e:square-sieve-parameter-bound} with $P=(S_1S_2)^{1/4}$.
Since $P^2=(S_1S_2)^{1/2}\leq S_1$, we have $S_1S_2/P=P^3\leq S_1P$ and $(S_1+S_2)P\leq2S_1P$.
\end{proof}

\subsection{Geometry of numbers}\label{s:lattice-counting}

For a rank two lattice $\Lambda\subseteq\mathbb R^2$, write $\det\Lambda$ for its determinant or covolume.
For $i=1,2$, its $i$-th successive minimum $\lambda_i(\Lambda)$ is the least $R>0$ for which $\Lambda$ contains $i$ linearly independent vectors of Euclidean norm at most $R$.
We choose a reduced basis $(\mathbf b_1,\mathbf b_2)$ with $\lVert\mathbf b_i\rVert\asymp\lambda_i(\Lambda)$.
The standard basis estimates~\cite[Section~3.2, pp.~72--74]{browning-cubic} give, for $\mathbf v=z_1\mathbf b_1+z_2\mathbf b_2$,
\begin{equation}\label{e:reduced-basis-bounds}
 \det\Lambda\asymp\lambda_1(\Lambda)\lambda_2(\Lambda),
 \qquad
 |z_i|\ll\frac{\lVert\mathbf v\rVert}{\lambda_i(\Lambda)},
 \quad i=1,2.
\end{equation}

We need bounds for the changes of basis to control the coefficients of auxiliary forms (see the proofs of Proposition~\ref{P:balanced-ranges} and Lemma~\ref{L:unbalanced-value-bound}).
Let $T$ be an invertible real $2\times2$ matrix, whose columns generate the lattice $T\ZZ^2$.
A reduced basis of this lattice has the form $TA$, with $A\in\mathrm{GL}_2(\ZZ)$.

For every nonzero $\mathbf n\in\ZZ^2$, we have
\[
 1\leq\lVert\mathbf n\rVert
 =\lVert T^{-1}(T\mathbf n)\rVert
 \leq\lVert T^{-1}\rVert\lVert T\mathbf n\rVert.
\]
Thus every nonzero vector of $T\ZZ^2$ has norm at least $\lVert T^{-1}\rVert^{-1}$.
The columns of $T$ give two linearly independent vectors in this lattice, each of norm at most $\lVert T\rVert$.
By the definition of the successive minima,
\[
 \lambda_1(T\ZZ^2)\geq\lVert T^{-1}\rVert^{-1},
 \qquad
 \lambda_2(T\ZZ^2)\leq\lVert T\rVert.
\]
Estimate~\eqref{e:reduced-basis-bounds} now gives
\begin{equation}\label{e:reduced-change-bounds}
 \begin{aligned}
 \lVert A\rVert+\lVert A^{-1}\rVert
 &\ll1+\lVert T\rVert\lVert T^{-1}\rVert,\\
 |z_i|&\ll\lVert\mathbf v\rVert\lVert T^{-1}\rVert,
 \qquad i=1,2,
 \end{aligned}
\end{equation}
when $\mathbf v=TA\mathbf z$.
Here we use $A=T^{-1}(TA)$ and $\lVert A^{-1}\rVert=\lVert A\rVert$ for $A\in\mathrm{GL}_2(\ZZ)$.

Let $\Lambda\subseteq\mathbb R^2$ be a lattice of rank two, and suppose that
\[
 \lambda_1(\Lambda)\geq1,\qquad R\geq2,\qquad M\geq1.
\]
For a nonzero vector $(x,y)$, we call $x/y$ its slope, taking the value $\infty$ when $y=0$.
Partition a fixed bounded range of the slope into intervals $I=[s_0,s_0+M^{-1})$.
For each interval, define $\Lambda_I$ as the image of $\Lambda$ under
\begin{equation}\label{e:intrinsic-interval-map}
 (x,y)\mapsto
 \left(\frac{M}{R^{1/2}}(x-s_0y),
       \frac{y}{R^{1/2}}\right).
\end{equation}
This map has determinant $M/R$ and sends points with $|y|\asymp R^{1/2}$ and $x/y\in I$ into a fixed bounded rectangle $\mathcal R$.
The intervals to be counted are
\[
 \mathcal I=\{I:\Lambda_I\cap(\mathcal R\setminus\{\mathbf0\})\ne\varnothing\}.
\]
We use the following bound to sum the point counts over these intervals.
\begin{lemma}\label{L:intrinsic-completion}
We have
\[
 \sum_{I\in\mathcal I}\frac1{\lambda_1(\Lambda_I)}
 \ll R^{1/2}\log(2R)
 +\left(\frac{RM}{\det\Lambda}\right)^{1/2}.
\]
\end{lemma}
For $I\in\mathcal I$, the lattice $\Lambda_I$ contains a nonzero vector in the fixed bounded rectangle $\mathcal R$, so $\lambda_1(\Lambda_I)\ll1$.
The lemma therefore also bounds the number of intervals in $\mathcal I$.
Both bounds hold for intervals of length $M^{-1}$ in a smooth function of the slope $x/y$, provided the absolute value of its derivative is bounded above and bounded away from zero.
The implied constants depend only on the range of $x/y$, $\mathcal R$, and the upper and lower bounds for the absolute value of the derivative.

\begin{proof}
A nonzero vector of $\Lambda$ is primitive if it is not an integer multiple of another lattice vector by a factor greater than one.
For $H\geq1$, applying~\cite[Lemma~2.7]{hb-squarefull} in lattice coordinates gives
\[
 \#\{\gamma\in\Lambda:\gamma\text{ primitive},\ \lVert\gamma\rVert\leq H\}
 \ll1+\frac{H^2}{\det\Lambda}.
\]

Group the intervals in $\mathcal I$ into dyadic ranges $1/\lambda_1(\Lambda_I)\sim T$, and choose a shortest nonzero vector in each $\Lambda_I$.
Its inverse image in $\Lambda$ is primitive, and the map~\eqref{e:intrinsic-interval-map} bounds its length by $H_1$, where
\[
 1\leq H_1\asymp\frac{R^{1/2}}{T}.
\]
For a fixed primitive $(x,y)\in\Lambda$ of length $r$, the possible interval endpoints satisfy $|x-s_0y|\ll H_1/M$.
Choose $C$ sufficiently large in terms of the fixed range of $x/y$.
If $|y|\geq r/C$, the possible left endpoints $s_0$ lie in an interval of length $O(H_1/(Mr))$, giving $O(1+H_1/r)$ grid intervals.
Otherwise $|x-s_0y|\gg r$, so $r\ll H_1/M$ and the total number of intervals, $O(M)$, satisfies the same bound.
Summing these multiplicities over the primitive vectors of length at most $H_1$ bounds the number of intervals in the chosen dyadic range.
Partial summation of the primitive lattice point bound, using $\lambda_1(\Lambda)\geq1$, gives
\[
 \sum_{\substack{\gamma\in\Lambda\ \mathrm{primitive}\\\lVert\gamma\rVert\leq H_1}}
 \left(1+\frac{H_1}{\lVert\gamma\rVert}\right)
 \ll H_1+\frac{H_1^2}{\det\Lambda}.
\]
By~\eqref{e:intrinsic-interval-map}, $\det\Lambda_I=(M/R)\det\Lambda$.
A nonzero vector of $\Lambda_I$ in $\mathcal R$ gives $\lambda_1(\Lambda_I)\ll1$, and~\eqref{e:reduced-basis-bounds} gives
\[
 T\gg\max\left(1,
 \left(\frac{R}{M\det\Lambda}\right)^{1/2}\right).
\]
Also $T\ll R^{1/2}$, since $\lambda_1(\Lambda)\geq1$, so there are $O(\log(2R))$ dyadic classes.
Multiplying the interval count by $T$ gives a contribution of
\[
 O\left(R^{1/2}+\frac{R}{T\det\Lambda}\right)
\]
from one dyadic class.
The first term sums to $O(R^{1/2}\log(2R))$, while the second is a geometric sum bounded by $O((RM/\det\Lambda)^{1/2})$ using the lower bound for $T$.

Under the stated change of parameter, each interval is covered by $O(1)$ intervals of length $M^{-1}$ in the slope $x/y$, and conversely.
For two such intervals $I,I'$ with left endpoints satisfying $|s_0-s_0'|\ll M^{-1}$, the change of normalized coordinates is the shear
\[
 (z_1,z_2)\mapsto\bigl(z_1+M(s_0-s_0')z_2,z_2\bigr).
\]
Since $M|s_0-s_0'|\ll1$, this map and its inverse have norm at most a constant $C\geq1$ depending only on the derivative bounds.
It follows that
\[
 C^{-1}\lambda_i(\Lambda_I)
 \leq\lambda_i(\Lambda_{I'})
 \leq C\lambda_i(\Lambda_I),
 \qquad i=1,2.
\]

The shear also sends $\mathcal R$ into a fixed bounded rectangle.
Applying the bound just proved with this enlarged rectangle, and using the bounded overlap of the interval coverings, proves the claim.
\end{proof}

\subsection{Binary quadratic equations}\label{s:quadratic-representations}

A factor of the auxiliary form $\Psi(d,e;u,v)$ from~\eqref{e:overview-auxiliary-equation} that depends only on $(u,v)$ fixes finitely many directions $(u:v)$.
For each corresponding primitive $(u,v)$, equation~\eqref{e:factor-equation} is a binary quadratic equation in $(d,e)$.
We use the following bound in Sections~\ref{s:balanced} and~\ref{s:unbalanced}.

Let $F$ be a primitive integral binary quadratic form, and write
\[
 F(x_1,x_2)=Ax_1^2+Bx_1x_2+Cx_2^2,
 \qquad D=B^2-4AC.
\]
Assume that $D>0$ and $D$ is not a square.
Let $w>1$ be the squarefree integer such that $\QQ(\sqrt D)=\QQ(\sqrt w)$.
For $m\in\ZZ\setminus\{0\}$ and $X\geq1$, put
\[
 N_F(m;X)=\#\{(x_1,x_2)\in\ZZ^2:|x_i|\leq X,\ F(x_1,x_2)=m\}.
\]
Khayutin, Nelson, and Steiner~\cite[(10.1)]{khayutin-nelson-steiner} proved
\[
 N_F(m;X)\ll\tau(m)\log(2X\lVert F\rVert).
\]
We record the following refinement.

\begin{lemma}\label{L:binary-quadratic-representations}
We have
\[
 N_F(m;X)\ll\tau(m)\frac{\log(2X\lVert F\rVert)}{\log(2w)},
\]
with an absolute implied constant.
\end{lemma}

\begin{proof}
We follow the proof of~\cite[(10.1)]{khayutin-nelson-steiner}.
Since $D$ is nonsquare, $A\ne0$.
In $L=\QQ(\sqrt w)$, put
\[
 \mathfrak b=\left(A,\frac{B+\sqrt D}{2}\right)\mathcal O_L,
 \qquad
 \beta=Ax_1+\frac{B+\sqrt D}{2}x_2.
\]
The element $(B+\sqrt D)/2$ is integral.
Using $\gcd(A,B,C)=1$ and $F(x_1,x_2)=m$, we obtain
\[
 \mathfrak b\bar{\mathfrak b}=(A),\qquad
 \Nm\mathfrak b=|A|,\qquad
 \Nm\beta=Am.
\]
Since $\beta\in\mathfrak b$, each representation determines an integral ideal $(\beta)\mathfrak b^{-1}$ of norm $|m|$.
There are at most $\tau(m)$ such ideals, and two elements determining the same ideal differ by a unit of norm $+1$.

For each ideal, fix one corresponding element $\beta_0$.
The others have the form
\[
 \beta=\pm\eta^k\beta_0,\qquad k\in\ZZ,
\]
where $\eta>1$ is the least totally positive unit of $L$.
By~\eqref{e:elementary-unit-lower-bounds}, applied to the fundamental discriminant of $L$, we have $\log\eta\gg\log(2w)$.
Since $\sqrt D\ll\lVert F\rVert$ and $|\Nm\beta|=|Am|\geq1$, the coordinate bounds imply
\[
 (X\lVert F\rVert)^{-1}\ll|\sigma(\beta)|\ll X\lVert F\rVert
\]
at both embeddings.
These bounds allow $O(\log(2X\lVert F\rVert)/\log(2w))$ integers $k$, since $w\leq D\ll\lVert F\rVert^2$.
There are at most two signs for each $k$.
The element $\beta$ determines $(x_1,x_2)$ uniquely, so summing over the ideals proves the lemma.
\end{proof}

\section{The approximate determinant method}\label{s:determinant-method}

\subsection{Reuss's approximate determinant method}\label{s:reuss-method}

We recall the determinant construction underlying Reuss's bound~\eqref{e:reuss-box-estimate} and record the normalization factors used below.
For a solution of~\eqref{e:factor-equation}, put
\[
 s=\frac de,\qquad t=\frac vu,
 \qquad
 t-s^2=\frac{h}{e^2u}.
\]
With $Y=DE\sqrt{UV}\asymp D^2U\asymp E^2V$ as in~\eqref{e:dyadic-size}, take $1\leq M\leq Y$ and partition the range of $s$ into intervals of length $D/(EM)$.
For an interval with right endpoint $s_0$, normalize the ratios by
\[
 S=\frac{s}{s_0},\qquad T=\frac{t}{s_0^2}.
\]
Since $s_0\asymp D/E$, equation~\eqref{e:factor-equation} gives
\begin{equation}\label{e:normalized-pell-relation}
 S=1+O(M^{-1}),\qquad
 T-S^2=\frac{h}{s_0^2e^2u}=O(Y^{-1}).
\end{equation}
Thus the normalized points lie close to the parabola $T=S^2$.

We seek a nonzero auxiliary form $\Psi(d,e;u,v)$ of bidegree $(A,B)$: homogeneous of degree $A$ in $(d,e)$ and of degree $B$ in $(u,v)$.
The positive integers $A,B$ will be chosen to make the relevant determinants vanish.
We evaluate the monomials $s^at^b$, with $0\leq a\leq A$ and $0\leq b\leq B$, at the points in this interval.
The number of monomials is
\[
 H=(A+1)(B+1).
\]
For any $H$ points, let $\Delta$ and $\Delta_0$ be the determinants formed from $s^at^b$ and $S^aT^b$, respectively.
Factoring $s_0^{a+2b}$ from the $(a,b)$-column of the matrix defining $\Delta$ gives
\begin{equation}\label{e:normalized-evaluation-determinant}
 \Delta=s_0^{H(A+2B)/2}\Delta_0.
\end{equation}
Put
\begin{equation}\label{e:integer-evaluation-determinant}
 \Delta_1
 =\left(\prod_{k=1}^H e_k^A u_k^B\right)\Delta.
\end{equation}
Multiplying row $k$ of the matrix defining $\Delta$ by $e_k^Au_k^B$ gives the integer entries $d_k^ae_k^{A-a}u_k^{B-b}v_k^b$, so $\Delta_1\in\ZZ$.
It therefore suffices to prove $|\Delta_1|<1$ for every choice of $H$ points, since then $\Delta_1=0$.
The evaluation matrix then has rank less than $H$, and a nontrivial integral relation among its columns gives the required auxiliary form.

To estimate $\Delta_0$, Reuss~\cite[Section~2.2.1]{reuss} expanded the entries in the local coordinates $S-1$ and $T-S^2$.
A Taylor monomial of degrees $i,j$ has size $O_{A,B}(M^{-i}Y^{-j})$.
By multilinearity, terms using the same Taylor monomial in two columns vanish.
Up to a factor depending on $A,B$, the determinant is bounded by the product of the $H$ largest such sizes, including the constant monomial.

Let $w_1\leq\cdots\leq w_H$ be the $H$ smallest weights $i\log M+j\log Y$ with $i,j\geq0$, counting distinct pairs separately even when their weights are equal.
When $\Delta_0\ne0$, this gives Reuss's determinant estimate
\begin{equation}\label{e:reuss-determinant-bound}
 \log|\Delta_0|\leq-\sum_{k=1}^H w_k+O_{A,B}(1).
\end{equation}

Since $s_0\asymp D/E$ and $D^2U\asymp E^2V$, the factors in~\eqref{e:normalized-evaluation-determinant} and~\eqref{e:integer-evaluation-determinant} contribute
\[
 \begin{aligned}
 &H(A\log E+B\log U)+\frac H2(A+2B)\log(D/E)\\
 &\qquad=\frac H2\bigl(A\log(DE)+B\log(UV)\bigr)+O_{A,B}(1).
 \end{aligned}
\]
Together with~\eqref{e:reuss-determinant-bound}, this gives
\begin{equation}\label{e:determinant-vanishing-criterion}
 \log|\Delta_1|\leq\frac H2(A\log(DE)+B\log(UV))-\sum_{k=1}^H w_k+O_{A,B}(1),
\end{equation}
whenever $\Delta_1\ne0$.
It therefore suffices to make the right hand side negative.
Reuss's optimization of the auxiliary degrees~\cite[Section~2.2.1, pp.~28--30]{reuss}, followed by counting on the resulting curves, gave the point count~\eqref{e:reuss-box-estimate} recalled in Section~\ref{s:dyadic-bounds}.

\subsection{An example}\label{s:determinant-example}

We first work out Reuss's bound~\eqref{e:reuss-determinant-bound} in the case $A=1$, $B=2$, and $M=Y$.
For six normalized points $(S_k,T_k)$, equation~\eqref{e:normalized-pell-relation} gives logarithmic coordinates
\begin{equation}\label{e:log-local-coordinates}
 \alpha_k=\log S_k=O(M^{-1}),\qquad
 \beta_k=\log(T_k/S_k^2)=O(M^{-1}).
\end{equation}
The determinant to be bounded is that of the matrix $\mathcal E$ with entries
\[
 \mathcal E_{k,(a,b)}
 =S_k^aT_k^b=\exp((a+2b)\alpha_k+b\beta_k),
\]
where $0\leq a\leq1$, $0\leq b\leq2$ and $1\leq k\leq6$.
Here the entrywise Taylor expansion is
\begin{equation}\label{e:log-taylor-expansion}
 \mathcal E_{k,(a,b)}
 =\sum_{i,j\geq0}\frac{\alpha_k^i\beta_k^j}{i!j!}(a+2b)^ib^j.
\end{equation}
The six largest Taylor monomials are those of total degree at most two.
Let $U_0$ be the $6\times6$ matrix whose $k$-th row is
\[
 \left(1,\alpha_k,\beta_k,\frac{\alpha_k^2}{2},
       \alpha_k\beta_k,\frac{\beta_k^2}{2}\right).
\]
For each $(a,b)$, put $N_1=a+2b$ and $N_2=b$.
Let $V_0$ be the $6\times6$ matrix whose column indexed by $(a,b)$ is
\[
 (1,N_1,N_2,N_1^2,N_1N_2,N_2^2)^t,
\]
with columns ordered by $(a,b)=(0,0),(1,0),(0,1),(1,1),(0,2),(1,2)$.
Their product is the Taylor polynomial of total degree two, entry by entry:
\[
 (U_0V_0)_{k,(a,b)}
 =\sum_{i+j\leq2}\frac{\alpha_k^i\beta_k^j}{i!j!}(a+2b)^ib^j.
\]

For an $H\times L$ matrix $U$ and an $L\times H$ matrix $V$, with $L\geq H$, the Cauchy--Binet formula states
\[
 \det(UV)=\sum_{\substack{\mathcal S\subseteq\{1,\ldots,L\}\\ \#\mathcal S=H}}
 \det U_{\mathcal S}\,\det V_{\mathcal S},
\]
where $U_{\mathcal S}$ consists of the columns of $U$ indexed by $\mathcal S$, and $V_{\mathcal S}$ of the corresponding rows of $V$, both in increasing index order.
Here $H=6$, and the intermediate indices are the pairs $(i,j)$ in the Taylor expansion.
Applying the Cauchy--Binet formula to finite truncations of~\eqref{e:log-taylor-expansion}, and then passing to the limit, gives
\[
 \det\mathcal E=\det U_0\,\det V_0+O(M^{-9}).
\]
Indeed, the six monomials used in $U_0$ have total degree $0+1+1+2+2+2=8$, and every other choice has total degree at least $9$.
The absolute convergence of~\eqref{e:log-taylor-expansion} bounds the sum of these remaining terms by $O(M^{-9})$.
The column sizes of $U_0$ and the fixed entries of $V_0$ give
\[
 \det U_0=O(M^{-8}),\qquad \det V_0=O(1).
\]
Using these two bounds recovers Reuss's estimate
\[
 \det\mathcal E\ll M^{-8}.
\]

However, $V_0$ is a singular matrix.
In fact, we have
\[
 V_0=
 \begin{pmatrix}
 1&1&1&1&1&1\\
 0&1&2&3&4&5\\
 0&0&1&1&2&2\\
 0&1&4&9&16&25\\
 0&0&2&3&8&10\\
 0&0&1&1&4&4
 \end{pmatrix}.
\]
Since $N_1-2N_2=a\in\{0,1\}$ in every column, we have
\begin{equation}\label{e:coefficient-grid-relation}
 (N_1-2N_2)(N_1-2N_2-1)
 =N_1^2-4N_1N_2+4N_2^2-N_1+2N_2=0.
\end{equation}
This is a linear relation among the rows of $V_0$, so $\det V_0=0$ and the contribution $\det U_0\det V_0$ vanishes identically.
The remainder bound now gives
\[
 \det\mathcal E\ll M^{-9}.
\]
This improves Reuss's bound by a factor of $M$.

\subsection{An improved determinant estimate}

For the reason explained in Section~\ref{s:main-counting-argument}, we fix $A=1$ and vary $B$.
In the preceding example, $M=Y$, so~\eqref{e:log-local-coordinates} bounds both logarithmic coordinates by $O(M^{-1})$.
We now introduce a separate parameter $X$ for the error in the quadratic approximation, and assume $|\alpha_k|\ll M^{-1}$ and $|\beta_k|\ll X^{-1}$.
We take $X=Y$ in the balanced ranges and $X=DW^{1/2}$ in Lemma~\ref{L:unbalanced-auxiliary}. 
The additional factor $|a|^{1/2}$ in the latter error bound is handled by Lemma~\ref{L:determinant-larger-error}.
The logarithmic coordinates~\eqref{e:log-local-coordinates} give
\[
 S_k=e^{\alpha_k},\qquad T_k=e^{2\alpha_k+\beta_k},\qquad
 S_k^aT_k^b=e^{(a+2b)\alpha_k+b\beta_k}.
\]
Thus we must bound the determinant of the values of the $2(B+1)$ monomials with $0\leq a\leq1$ and $0\leq b\leq B$ when $\alpha_k$ and $\beta_k$ are small on these two scales.

To use the same estimate in Section~\ref{s:unbalanced}, we allow complex logarithms, since the embedding ratios for an imaginary quadratic field lie on the unit circle.
We also allow row $k$ to be multiplied by $e^{-\sigma\alpha_k-\tau\beta_k}$.
The balanced application has $\sigma=\tau=0$, while normalizing by norms in the unbalanced ranges requires taking $\sigma=B+\tfrac12$ and $\tau=B/2$.
These nonzero row factors preserve the rank of the matrix.

Fix $C\geq1$, let $B\geq1$ be an integer, and recall that $H=2(B+1)$.
Let $M,X\geq2$ and $\sigma,\tau\in\tfrac12\ZZ$ satisfy
\begin{equation}\label{e:determinant-parameters}
 M<X,
 \qquad
 |\sigma|,|\tau|\leq2B+1.
\end{equation}
The condition $M<X$ makes the error scale smaller than the interval scale.
For $\alpha_k,\beta_k\in\mathbb C$, with $1\leq k\leq H$, define the $H\times H$ matrix $\mathcal E$ by its entries
\begin{equation}\label{e:general-log-evaluation-matrix}
 \mathcal E_{k,(a,b)}
 =
 \exp((a+2b-\sigma)\alpha_k+(b-\tau)\beta_k),
 \quad
 0\leq a\leq1,
 \quad
 0\leq b\leq B.
\end{equation}
Assume that the local coordinates satisfy
\begin{equation}\label{e:determinant-local-bounds}
 |\alpha_k|\leq CM^{-1},
 \qquad
 |\beta_k|\leq CX^{-1},
 \qquad 1\leq k\leq H.
\end{equation}
\begin{lemma}\label{L:quantitative-two-layer-determinant}
If $\det\mathcal E\ne0$, then
\[
 \log|\det\mathcal E|
 \leq
 -B(B+1)\log M-(B+1)\log X+\frac{(\log X)^2}{4\log M}+O(B^2),
\]
where the implied constant only depends on $C$.
\end{lemma}

\begin{proof}
First suppose that $\sigma=\tau=0$ and put
\[
 p=\log M,\qquad q=\log X.
\]
The coefficient grid is
\[
 \Gamma_B=\{(N_1,N_2)=(a+2b,b):0\leq a\leq1,\ 0\leq b\leq B\}.
\]
Give $N_1^iN_2^j$ the weight $ip+jq$, corresponding to the size $M^{-i}X^{-j}$ of its Taylor factor.
The relation~\eqref{e:coefficient-grid-relation} gives, on $\Gamma_B$,
\[
 N_2^2=N_1N_2-\frac14N_1^2+\frac14N_1-\frac12N_2.
\]
Since $q>p>0$, the four terms on the right have weights $p+q,2p,p,q$, all strictly smaller than $2q$.
Multiplying by $N_1^iN_2^{j-2}$ and repeating reduces every monomial with $j\geq2$ to a linear combination of monomials with $j=0,1$, without increasing the weight.
To span the coefficient vectors up to any weight, it is therefore enough to use monomials from the two rows $j=0,1$ up to that weight.
For $p,q>0$, define
\[
 W_B(p,q)
 =\min_{\substack{\mathcal S\subseteq\ZZ_{\geq0}\times\{0,1\}\\ \#\mathcal S=2(B+1)}}
 \sum_{(i,j)\in\mathcal S}(ip+jq).
\]
This is the sum of the $2(B+1)$ smallest weights in the two rows $j=0,1$, counting distinct pairs separately.
The reduction to the two rows shows that any $H$ monomials with independent coefficient vectors have total weight at least $W_B(p,q)$.
Indeed, order the selected vectors by weight.
The first $r$ independent vectors are spanned by monomials with $j=0,1$ of weight at most the $r$-th selected weight, so there must be at least $r$ such monomials.

To bound $W_B(p,q)$, fix the number $r$ of entries chosen from the row $j=1$.
The least total weight comes from the first $H-r$ entries with $j=0$ and the first $r$ entries with $j=1$.
Since $H=2(B+1)$, its weight is
\[
 \begin{aligned}
 &\frac p2((H-r)(H-r-1)+r(r-1))+rq\\
 &\qquad=B(B+1)p+(B+1)q-\frac{q^2}{4p}
       +p\left(B+1-r-\frac q{2p}\right)^2.
 \end{aligned}
\]
Dropping the square and minimizing over $0\leq r\leq H$ gives, for every $q>p$,
\begin{equation}\label{e:determinant-weight-lower-bound}
 W_B(p,q)\geq B(B+1)p+(B+1)q-\frac{q^2}{4p}.
\end{equation}
It remains to sum the Taylor expansion while controlling the dependence on $B$.
We apply the Cauchy--Binet formula to~\eqref{e:log-taylor-expansion}, now with $0\leq b\leq B$, and use the weight bound for every nonzero coefficient determinant.
For a set $\mathcal S$ of pairs $(i,j)$ with $\#\mathcal S=H$, define the matrices
\[
 (U_{\mathcal S})_{k,(i,j)}=\frac{\alpha_k^i\beta_k^j}{i!j!},
 \qquad
 (V_{\mathcal S})_{(i,j),(a,b)}=(a+2b)^ib^j.
\]
The pairs in $\mathcal S$ index the columns of $U_{\mathcal S}$ and the rows of $V_{\mathcal S}$.
If $\det V_{\mathcal S}\ne0$, the weight bound gives
\[
 M^{-\sum_{\mathcal S}i}X^{-\sum_{\mathcal S}j}
 \leq e^{-W_B(p,q)}.
\]
With $K=2C(B+1)$, the bounds $|a+2b|\leq2B+1$, $|b|\leq B$ and Hadamard's inequality give
\[
 |\det U_{\mathcal S}\det V_{\mathcal S}|
 \leq H^H M^{-\sum_{\mathcal S}i}X^{-\sum_{\mathcal S}j}
 \prod_{(i,j)\in\mathcal S}\frac{K^{i+j}}{i!j!}.
\]
Apply Cauchy--Binet first to finite Taylor truncations.
Since the exponential series converges absolutely, we can pass to the limit in the preceding bound and obtain
\[
 \begin{aligned}
 |\det\mathcal E|
 &\leq\sum_{\#\mathcal S=H}|\det U_{\mathcal S}\det V_{\mathcal S}|\\
 &\leq e^{-W_B(p,q)}\frac{H^H}{H!}
       \left(\sum_{i,j\geq0}\frac{K^{i+j}}{i!j!}\right)^H\\
 &\leq\exp(-W_B(p,q)+O_C(B^2)).
 \end{aligned}
\]
For general $\sigma,\tau$, the product of the additional row factors is
\[
 \exp\left(-\sigma\sum_k\alpha_k-\tau\sum_k\beta_k\right),
\]
whose absolute value has logarithm $O_C(B^2)$.
Thus, whenever $\det\mathcal E\ne0$, we have
\begin{equation}\label{e:determinant-weight-bound}
 \log|\det\mathcal E|\leq-W_B(p,q)+O_C(B^2).
\end{equation}
Together with~\eqref{e:determinant-weight-lower-bound}, this proves the lemma.
\end{proof}

For $B=2$ at the equal scales $M=X$ of Section~\ref{s:determinant-example}, put $p=\log M$.
The restriction $j\leq1$ gives
\[
 W_2(p,p)=0+p+p+2p+2p+3p=9p.
\]
With no restriction on $j$, the six smallest weights are $0,p,p,2p,2p,2p$, with total $8p$.
This explains the improvement from $M^{-8}$ to $M^{-9}$ in that example.

In Section~\ref{s:unbalanced}, the error in $\beta_k$ is larger by a factor $T$.
Use the same parameters~\eqref{e:determinant-parameters} and matrix~\eqref{e:general-log-evaluation-matrix}, but replace~\eqref{e:determinant-local-bounds} by
\[
 |\alpha_k|\leq CM^{-1},
 \quad
 |\beta_k|\leq CT/X,
 \quad
 1\leq T<X/M,
 \quad 1\leq k\leq H.
\]
\begin{lemma}\label{L:determinant-larger-error}
If $\det\mathcal E\ne0$, then
\[
 \begin{split}
 \log|\det\mathcal E|
 \leq{}&
 -B(B+1)\log M-(B+1)\log X\\
 &+\frac{(\log X)^2}{4\log M}+B\log T+O(B^2),
 \end{split}
\]
where the implied constant only depends on $C$.
\end{lemma}

\begin{proof}
As in the preceding proof, write $p=\log M$ and $q=\log X$, and put
\[
 q'=q-\log T=\log(X/T).
\]
Since $M<X/T$, estimate~\eqref{e:determinant-weight-bound} applies with $X/T$ in place of $X$ and gives
\[
 \log|\det\mathcal E|\leq-W_B(p,q')+O_C(B^2).
\]
A minimizing set at $q'$ has at most $B$ entries from the row $j=1$.
Indeed, if it had $r\geq B+1$ such entries, the largest selected weight in that row would be at least $Bp+q'$, while the first omitted weight with $j=0$ would be $(H-r)p$.
Since
\[
 Bp+q'>(B+1)p\geq(H-r)p,
\]
exchanging these entries would decrease the total weight.
Returning from $q'$ to $q$ adds $\log T$ to each of at most $B$ entries, so
\[
 W_B(p,q)\leq W_B(p,q')+B\log T.
\]
Combining this comparison with~\eqref{e:determinant-weight-lower-bound} proves the lemma.
\end{proof}

\subsection{Auxiliary forms}\label{s:binary-forms}

We bound the coefficients of the auxiliary forms and reduce the remaining count to square values of binary forms.

Let $B\geq1$ be an integer and recall that $H=2(B+1)$.
Choose finitely many integer tuples $(x_1,x_2,y_1,y_2)$ with both $(x_1,x_2)$ and $(y_1,y_2)$ primitive.
Define $\mathcal A$ to be the matrix with one row for each chosen tuple, whose entries are the values of the $H$ monomials
\[
 x_i y_1^{B-b}y_2^b,\qquad i=1,2,\quad 0\leq b\leq B.
\]
Suppose that $S,T\geq1$ and that every chosen tuple satisfies
\[
 |x_i|\leq S,\qquad |y_i|\leq T,\qquad i=1,2.
\]
\begin{lemma}\label{L:quantitative-two-layer-height}
If $\operatorname{rank}\mathcal A<H$, then there is a nonzero bihomogeneous form $\Psi$ with integer coefficients and bidegree at most $(1,B)$ that vanishes at every chosen tuple and satisfies
\[
 \log\lVert\Psi\rVert
 \leq
 H\left(\log(2S)+B\log(2T)+\tfrac12\log H\right).
\]
\end{lemma}

\begin{proof}
Put $r=\operatorname{rank}\mathcal A$.
If $r=0$, any of the evaluated monomials can serve as $\Psi$.
For $r\geq1$, choose $r$ independent rows and $r$ columns forming an invertible $r\times r$ submatrix, and adjoin one further column.
The signed $r\times r$ minors of the resulting $r\times(r+1)$ matrix give a nonzero integral relation among its columns.
Since the chosen rows span the row space of $\mathcal A$, extending this relation by zero gives the coefficients of a form $\Psi$ vanishing at every chosen tuple.
The entries of $\mathcal A$ are bounded by $ST^B$, so Hadamard's inequality gives
\[
 \lVert\Psi\rVert\leq r^{r/2}(ST^B)^r.
\]
Since $r\leq H-1$ and $S,T\geq1$, this gives the claimed coefficient bound.
\end{proof}

For a primitive factor $F$ of $\Psi$ with integer coefficients, Gelfond's factor inequality~\cite[Lemma~1.6.11, p.~27]{bg} and Lemma~\ref{L:quantitative-two-layer-height} give
\begin{equation}\label{e:quantitative-factor-height}
 \begin{split}
 \log\lVert F\rVert
 &\leq\log\lVert\Psi\rVert+O(B)\\
 &\leq H\left(\log(2S)+B\log(2T)+\tfrac12\log H\right)+O(B).
 \end{split}
\end{equation}

We now give the elimination argument described in Section~\ref{s:main-counting-argument}.
Let $G_0(x_1,x_2;y_1,y_2)$ be an integral bihomogeneous form of degree $2$ in $(x_1,x_2)$ and degree $1$ in $(y_1,y_2)$, and fix $h\in\ZZ\setminus\{0\}$.
We want to eliminate $(x_1,x_2)$ from the equations
\[
 G_0(x_1,x_2;y_1,y_2)=h,\qquad
 \Psi(x_1,x_2;y_1,y_2)=0,
\]
where $\Psi$ is a nonzero integral bihomogeneous form linear in $(x_1,x_2)$.
Factoring out the greatest common divisor of its two coefficient forms gives
\[
 \begin{aligned}
 \Psi(x_1,x_2;y_1,y_2)&=D_0(y_1,y_2)F(x_1,x_2;y_1,y_2),\\
 F(x_1,x_2;y_1,y_2)&=x_1F_1(y_1,y_2)+x_2F_2(y_1,y_2),
 \end{aligned}
\]
with $D_0,F$ integral, $F$ primitive, and $F_1,F_2$ coprime.
The nonzero forms among $F_1,F_2$ are homogeneous of a common degree $b\geq0$.
Solutions with $D_0(y_1,y_2)=0$ will be counted separately.
For the remaining equation $F=0$, define the binary form
\[
 K_0(y_1,y_2)=h\,G_0(F_2(y_1,y_2),-F_1(y_1,y_2);y_1,y_2).
\]
The following lemma reduces the count to the pairs $(y_1,y_2)$ for which $K_0(y_1,y_2)$ is a nonzero square.

\begin{lemma}\label{L:linear-elimination}
For each primitive $(y_1,y_2)$, the system
\[
 G_0(x_1,x_2;y_1,y_2)=h,\qquad
 F(x_1,x_2;y_1,y_2)=0
\]
has at most two primitive solutions $(x_1,x_2)$.
If a solution exists, then $K_0(y_1,y_2)$ is a nonzero square, and $K_0$ is a nonzero homogeneous form of odd degree $2b+1$.
\end{lemma}

\begin{proof}
A common rational projective zero of $F_1$ and $F_2$ would give a common linear factor over $\QQ$.
Since $F_1$ and $F_2$ are coprime, they cannot both vanish at a primitive pair $(y_1,y_2)$.
The equation $F(x_1,x_2;y_1,y_2)=0$ and the primitivity of $(x_1,x_2)$ therefore give a nonzero integer $g$ such that
\[
 (F_2(y_1,y_2),-F_1(y_1,y_2))=g(x_1,x_2).
\]
Since $G_0$ is quadratic in $(x_1,x_2)$, we obtain
\[
 \begin{aligned}
 K_0(y_1,y_2)&=h\,G_0(gx_1,gx_2;y_1,y_2)\\
 &=hg^2G_0(x_1,x_2;y_1,y_2)=(hg)^2\ne0.
 \end{aligned}
\]
For fixed $(y_1,y_2)$, the vector $(F_2(y_1,y_2),-F_1(y_1,y_2))$ determines $(x_1,x_2)$ up to sign, giving at most two choices.
The definition of $K_0$ gives homogeneity of degree $2b+1$, and the displayed identity shows that $K_0$ is nonzero whenever a solution exists.
\end{proof}

To apply the square sieve, we also need to bound the values of $K_0$ on the counting region.
Suppose that $K_0$ is nonzero.
For $T\geq1$, put
\begin{equation}\label{e:binary-form-value-size}
 W(K_0)=2+\max_{\substack{(y_1,y_2)\in\ZZ^2\\|y_i|\leq T}}|K_0(y_1,y_2)|.
\end{equation}
Then
\begin{equation}\label{e:quantitative-value-height}
 \begin{split}
 \log W(K_0)
 \leq{}&2\log\lVert F\rVert+\log|h|+\log\lVert G_0\rVert\\
       &+(2b+1)\log T+O(\log(2b+2)),
 \end{split}
\end{equation}
with an absolute implied constant.
Indeed, each of $F_1(y_1,y_2)$ and $F_2(y_1,y_2)$ has absolute value at most $(b+1)\lVert F\rVert T^b$, and $G_0$ has at most six monomials.

\section{The balanced ranges}\label{s:balanced}\label{s:pell-estimates}

We now bound the dyadic counting function $N(D,E,U,V)$ from~\eqref{e:dyadic-count} by counting points on the auxiliary forms from Section~\ref{s:determinant-method}.
We first prove a bound in terms of $M$, then choose $M$ and the auxiliary degree to obtain a power saving in the balanced ranges.
For $u\sim U$, $v\sim V$, and $\theta\geq3$, recall that a factorization is balanced if
\[
 \min(u,v)>(UV)^{1/(100\theta^2)},
\]
and unbalanced otherwise.

Fix a dyadic box as in~\eqref{e:scales}, with size $Y$ from~\eqref{e:dyadic-size}, an integer $B_0\geq1$, and parameters satisfying
\begin{equation}\label{e:balanced-parameter-assumptions}
 \begin{gathered}
 UV\geq2,\qquad \theta\geq3,\qquad Y\leq (UV)^{2\theta},\\
 1\leq M\leq\min(Y,UV),\qquad 1\leq B_0\leq4\theta^2,\qquad H_0=2(B_0+1).
 \end{gathered}
\end{equation}
Partition the ranges of $s=d/e$ and $t=v/u$ into intervals of lengths $D/(EM)$ and $V/(UM)$, respectively.
For each pair of intervals, let $s_0$ be the right endpoint of the interval in $s$ and $t_0$ the left endpoint of the interval in $t$.
The maps
\begin{equation}\label{e:balanced-interval-lattices}
 (d,e)\mapsto\left(\frac{M}{D}(d-es_0),\frac eE\right),\quad
 (u,v)\mapsto\left(\frac{M}{V}(v-ut_0),\frac uU\right)
\end{equation}
send the solutions in the chosen intervals into fixed bounded rectangles.
The images of $\ZZ^2$ under these maps are the corresponding interval lattices.

For a pair of intervals, an auxiliary form is a nonzero integral bihomogeneous form $\Psi$ of bidegree at most $(1,B_0)$ such that
\[
 \Psi(d,e;u,v)=0
\]
at every solution in that pair.
For every pair containing at least $H_0$ solutions, we require an auxiliary form whose coefficients in reduced bases for the lattices~\eqref{e:balanced-interval-lattices} satisfy
\begin{equation}\label{e:balanced-auxiliary-height-assumption}
 \log\lVert\Psi\rVert\ll\theta^4\log(2Y),
\end{equation}
with an absolute implied constant.
The square sieve in Corollary~\ref{C:anisotropic-square} gives the following bound.

\begin{proposition}\label{P:balanced-box}
Assume~\eqref{e:balanced-parameter-assumptions} and that every such pair of intervals has an auxiliary form satisfying~\eqref{e:balanced-auxiliary-height-assumption}.
Suppose also that
\begin{equation}\label{e:balanced-prime-condition}
 \left(\frac{UV}{M}\right)^{1/4}
 \geq C\theta^5\log(2UV),
\end{equation}
for a sufficiently large absolute constant $C$.
Then
\[
 N(D,E,U,V)
 \ll\theta^3\log^3(2UV)
 \left(\frac{UV}{M}\right)^{1/4}
 ((UVM)^{1/2}+U+V).
\]
The implied constant is absolute.
\end{proposition}

\begin{proof}
It suffices to prove the bound for each fixed $h\in\{\pm1,\pm2\}$.
By~\eqref{e:normalized-pell-relation}, we have
\[
 t=s^2+O\left(\frac{V}{UY}\right).
\]
Since $s\asymp D/E$ and $D^2/E^2\asymp V/U$, the map $s\mapsto s^2$ sends an interval of length $D/(EM)$ to one of length $\asymp V/(UM)$.
The error term is also $O(V/(UM))$ because $M\leq Y$, so each interval for either ratio is paired with only $O(1)$ intervals for the other.
There are $O(M)$ intervals in each range, hence $O(M)$ pairs that can contain solutions.
Those containing fewer than $H_0$ solutions contribute $O(H_0M)$ in total.

For each pair of intervals containing at least $H_0$ solutions, choose reduced bases for the two lattices $\Lambda_1,\Lambda_2$ in~\eqref{e:balanced-interval-lattices}.
Write $(x_1,x_2)$ and $(y_1,y_2)$ for the corresponding integer coordinates of $(d,e)$ and $(u,v)$, respectively.
Both are primitive, since $(d,e)$ and $(u,v)$ are primitive and the changes of integer coordinates are unimodular.
By~\eqref{e:reduced-basis-bounds}, these coordinates satisfy
\[
 |x_i|\ll\frac1{\lambda_i(\Lambda_1)},\qquad
 |y_i|\ll\frac1{\lambda_i(\Lambda_2)},\qquad i=1,2.
\]
In these coordinates the Pell equation and the auxiliary equation are
\[
 G_0(x_1,x_2;y_1,y_2)=h,
 \qquad \Psi(x_1,x_2;y_1,y_2)=0,
\]
where $G_0$ has bidegree $(2,1)$.
If $\Psi$ depends on $(x_1,x_2)$, apply Lemma~\ref{L:linear-elimination} and write
\[
 \Psi=D_0F,
\]
where $D_0$ is independent of $(x_1,x_2)$.
Otherwise, put $D_0=\Psi$ and $F=1$.
We first count the points with $D_0(y_1,y_2)=0$.
There are at most $B_0$ rational slopes $y_1/y_2$ in the reduced coordinates, each of which determines the original positive primitive $(u,v)$.
For fixed $(u,v,h)$, the remaining equation is
\[
 e^2v-d^2u=h.
\]
The form on the left is primitive with discriminant $4uv$, so Lemma~\ref{L:binary-quadratic-representations} applies.
Here $uv>1$ is squarefree, $uv\asymp UV$, the coefficients are $O(UV)$, and $d,e\ll (UV)^{2\theta}$ by~\eqref{e:balanced-parameter-assumptions}.
Since $\tau(h)=O(1)$, the lemma gives $O(\theta)$ solutions.
The points with $D_0(y_1,y_2)=0$ contribute $O(B_0\theta)$ per pair of intervals, so these points and those set aside at the start contribute $O(\theta^3M)$ points in total.

For the remaining points, Lemma~\ref{L:linear-elimination} gives a nonzero binary form $K_0$ of odd degree at most $2B_0+1$ such that
\[
 K_0(y_1,y_2)=\square\ne0.
\]
Each pair $(y_1,y_2)$ corresponds to at most two pairs $(x_1,x_2)$, so we can count these points with the square sieve once we bound the values of $K_0$.
The factor bound~\eqref{e:quantitative-factor-height} gives $\log\lVert F\rVert\ll\theta^4\log(2Y)$.
The maps~\eqref{e:balanced-interval-lattices} and their inverses have entries bounded by a fixed power of $Y$, so~\eqref{e:reduced-change-bounds} gives the same bound for the changes to reduced coordinates.
The coefficients of $G_0$ are therefore bounded by a fixed power of $Y$.
Using $Y\leq (UV)^{2\theta}$, we obtain
\[
 \log\left(\frac2{\lambda_1(\Lambda_2)}\right)+\log\lVert G_0\rVert+\log(2|h|)
 \ll \theta\log(2UV).
\]
To bound $W(K_0)$ from~\eqref{e:binary-form-value-size}, apply~\eqref{e:quantitative-value-height} with
\[
 T=\frac{C}{\lambda_1(\Lambda_2)},
\]
where $C$ is a sufficiently large absolute constant.
This gives
\[
 \log W(K_0)\ll \theta^5\log(2UV).
\]
If a pair of intervals contains only one primitive direction in $(y_1,y_2)$, it has at most two primitive $(y_1,y_2)$ and hence at most four solutions.
Otherwise two independent vectors in the bounded rectangle give $\lambda_2(\Lambda_2)\ll1$.
Choose intervals $I_1,I_2\subseteq[-T,T]$ containing the respective coordinates, with lengths
\[
 S_i=|I_i|\asymp\frac1{\lambda_i(\Lambda_2)},\qquad i=1,2,
 \qquad 1\leq S_2\leq S_1.
\]
The lattice determinant gives
\[
 S_1S_2\asymp\frac{1}{\lambda_1(\Lambda_2)\lambda_2(\Lambda_2)}
 \asymp\frac{UV}{M}.
\]
The bound for $W(K_0)$ and~\eqref{e:balanced-prime-condition} verify the hypotheses of Corollary~\ref{C:anisotropic-square}, which bounds the number of these solutions in one pair of intervals by
\begin{equation}\label{e:balanced-interval-count}
 O\left(\frac{\log^2(2UV)}{\lambda_1(\Lambda_2)}(UV/M)^{1/4}\right).
\end{equation}

To sum over the intervals in $t$, apply Lemma~\ref{L:intrinsic-completion} with
\[
 R=(U+V)^2,\qquad
 \Lambda=\operatorname{diag}\left(\frac{U+V}{V},\frac{U+V}{U}\right)\ZZ^2.
\]
This lattice satisfies $\lambda_1(\Lambda)\geq1$ and $\det\Lambda=R/(UV)$.
The vector
\[
 (x,y)=\left(\frac{U+V}{V}v,\frac{U+V}{U}u\right)
\]
has slope $x/y=(U/V)t\in[1/2,2]$.
Taking the interval endpoint in~\eqref{e:intrinsic-interval-map} to be $(U/V)t_0$ sends this vector to $(M(v-ut_0)/V,u/U)$, so $\Lambda_I$ is exactly the second lattice in~\eqref{e:balanced-interval-lattices}.
Since $U+V\leq UV+1$, the lemma gives
\[
 \sum_I\frac1{\lambda_1(\Lambda_I)}
 \ll (UVM)^{1/2}+(U+V)\log(2UV).
\]
Since each interval for $t$ corresponds to $O(1)$ intervals for $s$, this estimate sums~\eqref{e:balanced-interval-count} over all pairs of intervals.
The $O(1)$ points from each interval with only one primitive direction are also covered, since $M\leq UV$ and $1/\lambda_1(\Lambda_I)\gg1$.
The points with $D_0(y_1,y_2)=0$ and those in interval pairs containing fewer than $H_0$ solutions contribute only $O(\theta^3M)$, which is also covered because $M\leq UV$.
Summing over $h\in\{\pm1,\pm2\}$ completes the proof.
\end{proof}

We now choose the auxiliary degree and the interval parameter to obtain the power saving announced in Section~\ref{s:dyadic-bounds}.

\begin{proposition}\label{P:balanced-ranges}
Put $R=(UV)^{1/(100\theta^2)}$ and let $\eta>0$ be a sufficiently small absolute constant.
Assume that $UV$ is sufficiently large and
\[
 3\leq\theta\leq \eta\sqrt{\frac{\log(UV)}{\log\log(UV)}},
 \qquad U,V\geq R/4,
 \qquad Y\leq(4UV)^\theta.
\]
Then
\[
 N(D,E,U,V)
 \ll\theta^3\log^3(2UV)(UV)^{1-1/(200\theta^2)}.
\]
The implied constant is absolute.
\end{proposition}

\begin{proof}
For $DE\leq UV/R$, the elementary estimate~\cite[Section~2.6, p.~60]{reuss} gives
\[
 N(D,E,U,V)\ll (UV)^{1/2}+DE,
\]
which suffices.
We may therefore assume $DE>UV/R$ and set
\begin{equation}\label{e:balanced-localization}
 M=\left\lceil\frac{UV}{R^2}\right\rceil,
 \qquad B=\lceil\theta^2\rceil.
\end{equation}
We will apply Proposition~\ref{P:balanced-box} with $B_0=B$ and $H_0=2(B+1)$.
The restriction on $\theta$ gives
\[
 \theta^4\log(2\theta)\ll \eta^2\theta^2\log(UV).
\]
The hypothesis $Y\leq(4UV)^\theta$ also gives, uniformly in the stated range,
\[
 Y<(UV)^{2\theta},\qquad
 \frac{\log Y}{\log(UV)}
 \leq\theta+O\left(\frac{\theta}{\log(UV)}\right)
 \leq\frac{7\theta}{6}.
\]

To apply the determinant estimate, write
\[
 p=\frac{\log M}{\log(UV)},
 \qquad q=\frac{\log Y}{\log(UV)}.
\]
The preceding bounds and $DE>UV/R$ give $q>p$.
For $H_0$ points in one pair of intervals, use the logarithmic coordinates in~\eqref{e:log-local-coordinates}, which satisfy $|\alpha_k|\ll M^{-1}$ and $|\beta_k|\ll Y^{-1}$ by~\eqref{e:normalized-pell-relation}.
Lemma~\ref{L:quantitative-two-layer-determinant}, with $X=Y$ and $\sigma=\tau=0$, bounds the normalized determinant.
Restoring the factors in~\eqref{e:normalized-evaluation-determinant} and~\eqref{e:integer-evaluation-determinant} with $A=1$ adds $(B+1)(\log(DE)+B\log(UV))+O(B^2)$ to its logarithm, as in~\eqref{e:determinant-vanishing-criterion}.
Since $\log(DE)=(q-\tfrac12)\log(UV)$, every nonzero integer evaluation minor $\Delta_1$ satisfies
\begin{equation}\label{e:balanced-integer-minor}
 \log|\Delta_1|
 \leq
 -\Phi\log(UV)
 +O(B^2\log(2B)),
\end{equation}
where
\begin{equation}\label{e:determinant-gain}
 \Phi
 =\frac{B+1}{2}-\frac{q^2}{4p}
 -(1-p)B(B+1).
\end{equation}
Since $q\leq7\theta/6$, $p\geq9/10$, and $1-p\leq1/(50\theta^2)$, we have
\[
 \Phi\geq\frac{B+1}{2}-\frac{q^2}{4p}
 -\frac{B(B+1)}{50\theta^2}\geq\frac{\theta^2}{12},
\]
using $\theta^2\leq B<\theta^2+1$.
Taking $\eta$ sufficiently small therefore makes the right hand side of~\eqref{e:balanced-integer-minor} negative.
Thus all full evaluation minors vanish, and each pair of intervals containing at least $H_0$ points admits a nonzero auxiliary form of bidegree at most $(1,B)$.

To bound its coefficients, we apply Lemma~\ref{L:quantitative-two-layer-height} in reduced coordinates.
The inverse maps in~\eqref{e:balanced-interval-lattices} have norms $O(D+E)$ and $O(U+V)$, since $s_0\asymp D/E$, $t_0\asymp V/U$, and $M\geq1$.
Thus~\eqref{e:reduced-change-bounds} bounds the reduced coordinates by $O(D+E)$ and $O(U+V)$, respectively, and we have
\[
 \log(2(D+E))\ll \theta\log(UV),
 \qquad
 \log(2(U+V))\ll\log(UV).
\]
The unimodular changes of coordinates preserve the rank deficiency, so the lemma and~\eqref{e:quantitative-factor-height} give
\[
 \log\lVert\Psi\rVert\ll \theta^4\log(UV).
\]
The same bound holds for every primitive factor with integer coefficients after the changes of basis.
Since $DE>UV/R$ implies $Y>(UV)^{3/2-1/(100\theta^2)}$, this bound is $O(\theta^4\log(2Y))$, as required by Proposition~\ref{P:balanced-box}.
To check the prime condition~\eqref{e:balanced-prime-condition}, the choice of $M$ in~\eqref{e:balanced-localization} gives
\[
 \left(\frac{UV}{M}\right)^{1/4}\asymp R^{1/2}
 =(UV)^{1/(200\theta^2)}
 \geq(\log(UV))^{1/(200\eta^2)}.
\]
Since $\theta^5\log(2UV)\ll(\log(UV))^{7/2}$, decreasing $\eta$ if necessary verifies~\eqref{e:balanced-prime-condition}.
Since $U,V\geq R/4$, we have $U+V\leq8UV/R$, so Proposition~\ref{P:balanced-box} now gives
\[
 \begin{aligned}
 N(D,E,U,V)
 &\ll\theta^3\log^3(2UV)R^{1/2}
 ((UVM)^{1/2}+U+V)\\
 &\ll\theta^3\log^3(2UV)\frac{UV}{R^{1/2}}.
 \end{aligned}
\]
This concludes the proof.
\end{proof}

\section{Factorization in quadratic fields}\label{s:quadratic-preliminaries}

In the unbalanced case, fixing the smaller squarefree factor leads to the equation
\(
 d^2w=n^2+a
\)
from~\eqref{e:unbalanced-equation}.
We construct primitive $(x_1,x_2)$ and $(y_1,y_2)$ for these solutions so that the approximate determinant method and square sieve can also be applied in Section~\ref{s:unbalanced}.
Lemma~\ref{L:embedding-coordinates} supplies the parameters for the determinant method; Lemmas~\ref{L:embedding-lattices} and~\ref{L:embedding-intervals} give the coordinate bounds and interval estimates for the lattice count.
For each fixed pair $(y_1,y_2)$, Lemma~\ref{L:quadratic-unit-count} bounds the number of pairs $(x_1,x_2)$ satisfying the resulting binary quadratic equation.

\subsection{Integral coordinates}\label{s:quadratic-coordinates}

Let $a\in\ZZ\setminus\{0,-1\}$ be squarefree and put
\[
 K=\QQ(\sqrt{-a}).
\]
The equation $d^2w=n^2+a$ can be written in $K$ as
\[
 \Nm_{K/\QQ}(n+\sqrt{-a})
 =(n+\sqrt{-a})(n-\sqrt{-a})
 =d^2w.
\]
This factorization gives an equation of bidegree $(2,1)$ in the primitive $(x_1,x_2)$ and $(y_1,y_2)$.
On each irreducible auxiliary factor linear in the first pair, Lemma~\ref{L:linear-elimination} reduces the count to nonzero square values of an odd-degree binary form in the second.

We work in the ring of integers $\mathcal O_K$, using the two real embeddings if $a<0$ and the real and imaginary parts of one complex embedding if $a>0$.
For a solution of~\eqref{e:unbalanced-equation} with $d,w,n\in\NN$ and $w$ squarefree, we have $\gcd(n,d)=1$, since $\gcd(n,d)^2\mid a$ and $a$ is squarefree.
Put
\[
 d_2=\gcd(d,2),\qquad d'=\frac{d}{d_2},\qquad
 \zeta=\frac{n+\sqrt{-a}}{d_2}.
\]
The element $\zeta$ is integral: when $d$ is even, we have $n$ odd and $-a\equiv1\;(\mathrm{mod}\,4)$.

Choose $\omega$ so that $\mathcal O_K=\ZZ[\omega]$ and
\[
 \omega^2=b\omega+c,\qquad b\in\{0,1\},\quad c\in\ZZ.
\]
Write
\[
 \zeta=z_1+z_2\omega,\qquad z_1,z_2\in\ZZ.
\]
The coefficient of $\sqrt{-a}$ in $\omega$ is $\pm1$ or $\pm1/2$, so $0<|z_2|\leq2$.
The coefficient of $\omega$ in $(x_1+x_2\omega)^2(y_1+y_2\omega)$ has bidegree $(2,1)$ and equals
\[
 P_1(x_1,x_2)y_1+Q_1(x_1,x_2)y_2,
\]
where
\begin{equation}\label{e:square-map-coefficient-forms}
 P_1=2x_1x_2+bx_2^2,\qquad
 Q_1=x_1^2+2bx_1x_2+(c+b^2)x_2^2.
\end{equation}

The next lemma constructs algebraic integers $\mu,\nu$ from the ideal factorization of $(\zeta)$.
Let $r,\ell>0$ be the greatest common divisors of their respective coordinate pairs, and write
\[
 \alpha=x_1+x_2\omega,\qquad \beta=y_1+y_2\omega,
 \qquad \mu=r\alpha,\qquad \nu=\ell\beta
\]
where $(x_1,x_2)$ and $(y_1,y_2)$ are both primitive.
\begin{lemma}\label{L:square-map-coordinates}
There are positive integers $k,r,\ell$ and primitive $(x_1,x_2)$ and $(y_1,y_2)$ such that
\begin{equation}\label{e:square-map-elements}
 \mu^2\nu=k^2\zeta,
 \quad |\Nm\mu|=d'k,
 \quad \Nm\nu=wk^2,
 \quad k\ll|a|^{1/2}.
\end{equation}
Moreover,
\begin{equation}\label{e:square-map-integral-equation}
 P_1(x_1,x_2)y_1+Q_1(x_1,x_2)y_2=\xi, \quad \text{where}~ \xi=\frac{k^2z_2}{r^2\ell}\in\ZZ\setminus\{0\}.
\end{equation}
\end{lemma}

\begin{proof}
The norm of $\zeta$ is
\[
 \Nm_{K/\QQ}\zeta
 =\frac{n^2+a}{d_2^2}
 =(d')^2w.
\]
The ideals $(\zeta)$ and $(\bar\zeta)$ are coprime at every prime above $d'$, since $\zeta+\bar\zeta=2n/d_2$ is coprime to $d'$.
For each prime power $p^e\Vert d'$, coprimality and the norm identity give a unique prime ideal $\mathfrak p$ of norm $p$ with $\mathfrak p^{2e}\mid(\zeta)$.
Taking the product of these $\mathfrak p^e$ gives a unique integral ideal $\mathfrak d$ with
\[
 \Nm\mathfrak d=d',
 \qquad \mathfrak d^{\,2}\mid(\zeta).
\]
The lattice of $\mathfrak d$ has determinant $\asymp d'|\disc K|^{1/2}$, so~\eqref{e:reduced-basis-bounds} gives a nonzero element $\mu\in\mathfrak d$ with
\begin{equation}\label{e:ideal-generator-embedding-bound}
 |\sigma(\mu)|\ll (d')^{1/2}|\disc K|^{1/4}
\end{equation}
at every embedding $\sigma$.
Fix one such element $\mu$ for each ideal $\mathfrak d$, and put
\[
 k=\frac{|\Nm\mu|}{d'},
 \qquad \nu=\frac{k^2\zeta}{\mu^2}.
\]
Here $k=[\mathfrak d:(\mu)]$ is a positive integer.
The inclusions $k\mathfrak d\subseteq(\mu)$ and $(\zeta)\subseteq\mathfrak d^{\,2}$ show that $\nu$ is integral.
The definition of $\nu$ and the norm identities give the first three relations in~\eqref{e:square-map-elements}, while the embedding bound gives $k\ll |a|^{1/2}$.

Express $\mu$ and $\nu$ in the integral basis $(1,\omega)$ and divide each coordinate pair by the greatest common divisor of its entries to write
\[
 \mu=r(x_1+x_2\omega),\qquad
 \nu=\ell(y_1+y_2\omega),
\]
with $\gcd(x_1,x_2)=\gcd(y_1,y_2)=1$.
Comparing coefficients of $\omega$ in~\eqref{e:square-map-elements} gives
\[
 r^2\ell\bigl(P_1(x_1,x_2)y_1+Q_1(x_1,x_2)y_2\bigr)=k^2z_2.
\]
Thus $r^2\ell\mid k^2z_2$, and $\xi$ is a nonzero integer.
\end{proof}

We next bound the embeddings of $\alpha$ and $\beta$ and check that their coordinates determine the original solution.
Fix the integral basis for each $a$ before applying Lemma~\ref{L:square-map-coordinates}.
Put
\[
 \Delta=|\disc K|,\qquad |a|\leq\Delta\leq4|a|.
\]
The absolute norms of $\alpha$ and $\beta$ are
\begin{equation}\label{e:invariant-scalar-data}
 E_1=|\Nm\alpha|=\frac{d'k}{r^2},\qquad
 E_2=|\Nm\beta|=w\left(\frac{k}{\ell}\right)^2.
\end{equation}
\begin{lemma}\label{L:square-map-uniformity}
Fix $a,k,r,\ell$ as in Lemma~\ref{L:square-map-coordinates}.
Each choice of the primitive $(x_1,x_2)$ and $(y_1,y_2)$ corresponds to at most one solution $(d,w,n)$ of~\eqref{e:unbalanced-equation}.
If $n\geq2|a|^{1/2}$, there is an absolute constant $C$ such that, at every embedding $\sigma:K\hookrightarrow\mathbb C$,
\[
\begin{aligned}
 |a|^{-C}E_1^{1/2}
 &\leq|\sigma(\alpha)|
 \leq|a|^{C}E_1^{1/2},\\
 |a|^{-C}E_2^{1/2}
 &\leq|\sigma(\beta)|
 \leq|a|^{C}E_2^{1/2}.
\end{aligned}
\]
\end{lemma}

\begin{proof}
If $K$ is imaginary, the embeddings of $\alpha$ and $\beta$ have absolute values $E_1^{1/2}$ and $E_2^{1/2}$, respectively, so the bounds hold immediately.

If $K$ is real, both embeddings of $\mu$ satisfy~\eqref{e:ideal-generator-embedding-bound}, and their product has absolute value $d'k$.
Using $1\leq k\ll|a|^{1/2}$ and dividing by $r$ gives the first pair of inequalities.
When $n\geq2|a|^{1/2}$, the two embeddings of $\zeta$ have comparable size, so $\mu^2\nu=k^2\zeta$ gives the second pair.

For uniqueness in either case, the two coordinate pairs determine
\[
 \zeta=\frac{r^2\ell}{k^2}\alpha^2\beta.
\]
Its coefficient of $\sqrt{-a}$ is $1/d_2$, which determines $d_2$.
The constant term then determines $n$.
The norm identities in~\eqref{e:square-map-elements} recover $d'$ and $w$, and finally $d=d_2d'$.
\end{proof}

\subsection{Embedding coordinates and quadratic equations}\label{s:embedding-coordinates}\label{s:binary-quadratic-count}

We now choose parameters for the determinant method and coordinates for the lattice count in Section~\ref{s:unbalanced}.
Fix a squarefree integer $a\ne0,-1$ and consider positive integer solutions $(d,w,n)$ of~$d^2w=n^2+a$, with $w$ squarefree and $n\geq2|a|^{1/2}$.
Fix the values $k,r,\ell$ in their representations from Lemma~\ref{L:square-map-coordinates}.
Let $\alpha,\beta$ be the corresponding elements, with absolute norms $E_1,E_2$ from~\eqref{e:invariant-scalar-data}, and write
\begin{equation}\label{e:embedding-ratios}
 \rho_1=\frac{\alpha}{\bar\alpha},\qquad
 \rho_2=\frac{\bar\beta}{\beta}.
\end{equation}

\begin{lemma}\label{L:embedding-coordinates}
These solutions $(d,w,n)$ can be partitioned into $O(\log(2|a|))$ sets.
On each set, fixed nonzero multiples of $\rho_1,\rho_2$ have the form $e^{t_1},e^{t_2}$ when $K$ is real, and $e^{it_1},e^{it_2}$ when $K$ is imaginary.
The real parameters $t_1,t_2$ lie in fixed bounded intervals and satisfy
\[
 t_2=2t_1+O(|a|^{1/2}/n).
\]
The implied constants are absolute.
\end{lemma}

\begin{proof}
Taking conjugates in $\alpha^2\beta=k^2\zeta/(r^2\ell)$ gives
\begin{equation}\label{e:invariant-square-map}
 \rho_1^{-2}\rho_2=\frac{\bar\zeta}{\zeta}
 =1-\frac{\zeta-\bar\zeta}{\zeta}.
\end{equation}

When $K$ is real, Lemma~\ref{L:square-map-uniformity} gives $|a|^{-O(1)}\ll|\rho_1|\ll |a|^{O(1)}$.
Fix the sign of $\rho_1$ and a dyadic range $q\leq|\rho_1|<2q$; there are $O(\log(2|a|))$ choices.
Since $\Nm\beta>0$, we have $\rho_2>0$ and may define
\begin{equation}\label{e:normalized-embedding-coordinates}
 t_1=\log\frac{|\rho_1|}{q},\qquad
 t_2=\log\frac{\rho_2}{q^2}.
\end{equation}
By~\eqref{e:invariant-square-map},
\[
 t_2-2t_1=\log\frac{\bar\zeta}{\zeta}
 =\log\frac{n-\sqrt{|a|}}{n+\sqrt{|a|}}.
\]
Since $n\geq2\sqrt{|a|}$, the ratio inside the logarithm lies in $[1/3,1]$ and $|t_2-2t_1|\ll\sqrt{|a|}/n$.
In particular, the parameters lie in the fixed intervals
\[
 0\leq t_1<\log2,\qquad -\log3\leq t_2<2\log2.
\]

When $K$ is imaginary, the ratios $\rho_1=\alpha/\bar\alpha$ and $\rho_2=\bar\beta/\beta$ have absolute value $1$.
Since $\zeta=(n+i\sqrt a)/d_2$, the ratio $\bar\zeta/\zeta$ has principal argument $\varepsilon=-2\arctan(\sqrt a/n)$.
In particular,
\[
 \frac{\bar\zeta}{\zeta}=e^{i\varepsilon},\qquad
 |\varepsilon|\leq\frac{2\sqrt a}{n}\leq1.
\]
Choose the arguments of $\rho_1,\rho_2$ by putting
\begin{equation}\label{e:angular-embedding-coordinates}
 t_1=\arg\rho_1\in[-\pi,\pi),\qquad t_2=2t_1+\varepsilon.
\end{equation}
Equation~\eqref{e:invariant-square-map} gives $e^{it_2}=\rho_1^2e^{i\varepsilon}=\rho_2$, so $t_2$ is indeed an argument of $\rho_2$.
Since $|t_2|\leq2\pi+1$, both parameters range over fixed bounded intervals.
Partition these intervals into $O(1)$ subintervals of length at most $\pi/4$ and work with one pair at a time.
Each subinterval corresponds to an arc of the unit circle on which the chosen argument is continuous.
\end{proof}

We next choose the coordinates for the lattice count.
Fix one set constructed in the proof of Lemma~\ref{L:embedding-coordinates}.
Write $T_1,T_2$ for invertible real changes from the integral coordinates of $\alpha,\beta$, respectively, and put $\Lambda_i=T_i\ZZ^2$ for $i=1,2$.

\begin{lemma}\label{L:embedding-lattices}
The matrices $T_1,T_2$ can be chosen independently of the solution in this set.
For $i=1,2$, the slope $x/y=s_i(t_i)$ of the image $(x,y)$ is smooth and
\[
 |y|\asymp E_i^{1/2},\qquad
 |s_i(t_i)|\ll1,\qquad |s_i'(t_i)|\asymp1.
\]
Here $s_i'$ is the derivative of $s_i$ with respect to $t_i$.
We also have
\[
 \det\Lambda_i\asymp\Delta^{1/2},\qquad
 \lambda_1(\Lambda_i)\geq1,\qquad
 \lVert T_i\rVert+\lVert T_i^{-1}\rVert\ll |a|.
\]
All implied constants are absolute.
\end{lemma}

\begin{proof}
When $K$ is real, use the dyadic scale $q$ from~\eqref{e:normalized-embedding-coordinates} and rescale the embedding coordinates by
\[
 (\alpha,\bar\alpha)\mapsto
 (q^{-1/2}\alpha,q^{1/2}\bar\alpha),\qquad
 (\beta,\bar\beta)\mapsto(q\beta,q^{-1}\bar\beta).
\]
The slopes $x/y$ are $s_1(t_1)=\operatorname{sgn}(\rho_1)e^{t_1}$ for $\alpha$ and $s_2(t_2)=e^{-t_2}$ for $\beta$.
Thus, for $i=1,2$,
\[
 |s_i(t_i)|=|s_i'(t_i)|\asymp1.
\]
Each rescaling preserves the product of the embeddings, so $|xy|=E_i$.
Since $|y|^2=|xy|/|s_i(t_i)|$, we obtain $|y|\asymp E_i^{1/2}$.

When $K$ is imaginary, the parameter intervals chosen in the preceding proof have length at most $\pi/4$.
Use the real and imaginary parts of each element as coordinates, rotating on each fixed arc so that $|y|\asymp E_i^{1/2}$ and $|s_i(t_i)|=|x/y|\ll1$.
Here $s_1(t_1)=\cot(t_1/2+\phi_1)$ and $s_2(t_2)=\cot(-t_2/2+\phi_2)$, where $\phi_1,\phi_2$ are fixed rotation angles.
Hence, for $i=1,2$,
\[
 |s_i'(t_i)|=\frac{1+s_i(t_i)^2}{2}\asymp1.
\]

These rescalings or rotations are fixed on the chosen set.
They define $T_1,T_2$ and give the claimed bounds for $|y|$, $s_i$, and $s_i'$.
The real embedding lattice has determinant $\Delta^{1/2}$, and the complex embedding lattice has determinant $\Delta^{1/2}/2$.
The rescalings and rotations have determinant one, so these determinants are unchanged.
Every nonzero algebraic integer has nonzero integer norm, which bounds the length of its image below by $1$.

The unscaled embedding matrices have norm $O(|a|^{1/2})$, and their inverses have norm $O(1)$.
When $K$ is real, the bound~\eqref{e:ideal-generator-embedding-bound} and $|\Nm\mu|=d'k$ give
\[
 \max(q,q^{-1})
 \ll\max(|\rho_1|,|\rho_1|^{-1})
 \ll\frac{\Delta^{1/2}}k\ll |a|^{1/2}.
\]
Thus the diagonal rescalings and their inverses have norm $O(|a|^{1/2})$, while the rotations in the imaginary case have norm one.
\end{proof}

The next lemma gives the lattice sum needed for the square sieve.
We work within one set in the partition of solutions given by Lemma~\ref{L:embedding-coordinates}, with the lattices and slope functions from Lemma~\ref{L:embedding-lattices}.
For $M\geq1$, each parameter range is covered by a grid of intervals of length $M^{-1}$, starting at its left endpoint.
For $i=1,2$, let $J_i$ denote one of these intervals in the range of $t_i$.

We consider solutions $(d,w,n)$ of $d^2w=n^2+a$ with $w\sim W$ and $n\geq M|a|^{1/2}$, where $W\geq2$.
The norm formula~\eqref{e:invariant-scalar-data} for $\beta=y_1+y_2\omega$ gives
\[
 R=W(k/\ell)^2\asymp E_2.
\]
Writing $t_0$ for the left endpoint of $J_2$, we define $\Lambda_{J_2}$ as the image of $\Lambda_2$ under
\[
 (x,y)\mapsto
 \left(\frac{M}{R^{1/2}}(x-s_2(t_0)y),\frac{y}{R^{1/2}}\right),
\]
as in~\eqref{e:intrinsic-interval-map}.

\begin{lemma}\label{L:embedding-intervals}
We have
\[
 \sum_{J_2}\frac1{\lambda_1(\Lambda_{J_2})}
 \ll R^{1/2}\log(2R)
 +\left(\frac{RM}{\Delta^{1/2}}\right)^{1/2},
\]
where the sum is over intervals $J_2$ containing the parameter $t_2$ of at least one such solution.
The implied constant is absolute.
\end{lemma}

\begin{proof}
Lemma~\ref{L:embedding-coordinates} gives $|t_2-2t_1|\ll M^{-1}$.
Thus, for each fixed $J_1$, only $O(1)$ intervals $J_2$ can occur with $(t_1,t_2)\in J_1\times J_2$, and conversely.
There are $O(M)$ intervals in each parameter range, hence $O(M)$ pairs $(J_1,J_2)$ containing solutions.
We will use these counts in Section~\ref{s:unbalanced}.

Squarefreeness of $w$ and $\ell^2\Nm\beta=wk^2$ give
\begin{equation}\label{e:embedding-scale-lower-bound}
 \ell\mid k,\qquad R\geq W\geq2.
\end{equation}
By Lemma~\ref{L:embedding-lattices}, a solution with $t_2\in J_2$ satisfies
\[
 |y|\asymp R^{1/2},\qquad
 |x-s_2(t_0)y|\ll R^{1/2}/M.
\]
Its image in $\Lambda_{J_2}$ thus lies in a fixed bounded rectangle.
Lemma~\ref{L:embedding-lattices} gives $\lambda_1(\Lambda_2)\geq1$, $\det\Lambda_2\asymp\Delta^{1/2}$, and $|s_2'(t_2)|\asymp1$.
Applying the change of parameter form of Lemma~\ref{L:intrinsic-completion} gives the stated bound.
\end{proof}

An auxiliary factor depending only on $(y_1,y_2)$ fixes finitely many primitive pairs, so we also need a count with $\beta$ fixed.
For each such pair, equation~\eqref{e:square-map-integral-equation} is a binary quadratic equation in $(x_1,x_2)$.
Let $(y_1,y_2)$ be primitive, with $\beta=y_1+y_2\omega$ satisfying
\[
 \Nm\beta=wf^2,\qquad w>1\ \text{squarefree},\qquad f\in\NN.
\]
For $\xi\in\ZZ\setminus\{0\}$ and $X\geq1$, let $N_\beta(\xi;X)$ count primitive $(x_1,x_2)$ with $|x_i|\leq X$ satisfying
\begin{equation}\label{e:quadratic-count-equation}
 P_1(x_1,x_2)y_1+Q_1(x_1,x_2)y_2=\xi.
\end{equation}

\begin{lemma}\label{L:quadratic-unit-count}
We have
\[
 N_\beta(\xi;X)\ll\tau(\xi)
 \frac{\log\bigl(2|a|X(1+|y_1|+|y_2|)\bigr)}{\log(2w)}.
\]
The implied constant is absolute.
\end{lemma}

\begin{proof}
Let $F$ be the quadratic form on the left of~\eqref{e:quadratic-count-equation}.
Recall that $\omega^2=b\omega+c$, with $b\in\{0,1\}$.
By~\eqref{e:square-map-coefficient-forms}, the discriminant of $F$ is
\[
 4(y_1^2+by_1y_2-cy_2^2)
 =4\Nm(y_1+y_2\omega)=4wf^2>0.
\]
The coefficients of $x_1^2$ and $x_1x_2$ in $F$ are $y_2$ and $2(y_1+by_2)$, respectively.
Since $\gcd(y_1,y_2)=1$, the content $g$ of $F$ divides $2$.
If $g\nmid \xi$, there are no solutions.
Otherwise $F/g$ is primitive, with positive nonsquare discriminant and coefficient bound
\[
 \disc(F/g)=\frac{4wf^2}{g^2},
 \qquad \lVert F/g\rVert\ll |a|(1+|y_1|+|y_2|).
\]
Applying Lemma~\ref{L:binary-quadratic-representations} to $F/g$ with $m=\xi/g$ gives the claimed bound.
\end{proof}

\section{The unbalanced ranges}\label{s:unbalanced}

We now use the results from Section~\ref{s:quadratic-preliminaries} to count the factorizations~\eqref{e:factor-equation} for which one squarefree factor is small.
Let $w=uv$ with $u\sim U$, $v\sim V$, $UV\geq3$, and $\theta\geq3$, and suppose that $\min(u,v)\leq (UV)^{1/(100\theta^2)}$.

\subsection{Bounds for the counting function}

Interchanging the factors if necessary so that $v$ is the smaller one, the reduction~\eqref{e:unbalanced-equation} gives
\[
 n=ev,\qquad a=-hv,\qquad d^2w=n^2+a.
\]
For the factorizations of fundamental solutions in Section~\ref{s:pell-factorizations}, $a$ is squarefree, since $v$ is squarefree and is odd when $|h|=2$.
If $a=-1$, then $(h,v)=(1,1)$, so $e^2-ud^2=1$ and $(e+d\sqrt u)^2=\varepsilon_u$.
This gives a solution with $1<e+d\sqrt u<\varepsilon_u$, which contradicts the minimality of the fundamental solution.

For $a\in\ZZ\setminus\{0\}$ and $D,W>0$, recall from~\eqref{e:unbalanced-count} that
\[
 N_a(D,W)=\#\left\{(d,w,n)\in\ZZ_{>0}^3:
 \begin{array}{l}
 d^2w=n^2+a,\quad d\sim D,\\
 w\sim W,\quad w\ \text{squarefree}
 \end{array}\right\}.
\]
For fixed $a$, the pair $(d,w)$ determines $n$.
The range of $w$ in a box~\eqref{e:scales} is covered by taking $W=UV$ and $W=2UV$.

When $D$ is small, we can sum over $d$ and count the possible values of $n$ by the congruence $n^2\equiv-a\pmod{d^2}$.
\begin{lemma}\label{L:direct-count}
For squarefree $a\ne0$, $D\geq1$, and $W>0$,
\[
 N_a(D,W)\ll(W^{1/2}+D)\log(2D),
\]
with an absolute implied constant.
\end{lemma}

\begin{proof}
Since $a$ is squarefree, we have $\gcd(a,d)=1$ when $d^2\mid n^2+a$.
Write $\omega(d)$ for the number of distinct prime divisors of $d$.
Hensel's lemma and the Chinese remainder theorem give at most $4\cdot2^{\omega(d)}$ roots modulo $d^2$, with at most two roots at each odd prime power and at most four at the power of $2$.
For fixed $d$ and $w\sim W$, the values of $n^2$ lie in an interval of length $d^2W$, so the positive values of $n$ lie in an interval of length at most $dW^{1/2}$.
Each root therefore contributes $O(W^{1/2}/d+1)$ possible values of $n$.
Summing over $d\sim D$ gives
\[
 N_a(D,W)
 \ll\left(\frac{W^{1/2}}D+1\right)\sum_{d\leq2D}2^{\omega(d)}
 \ll(W^{1/2}+D)\log(2D).\qedhere
\]
\end{proof}

For larger $D$, we instead use the coordinates of Lemma~\ref{L:square-map-coordinates} in $K=\QQ(\sqrt{-a})$.
Let $a\ne0,-1$ be squarefree and suppose that
\begin{equation}\label{e:unbalanced-parameter-assumptions}
 \begin{gathered}
 W\geq3,\qquad
 3\leq\theta\leq c_1\sqrt{\frac{\log W}{\log\log W}}, \\
 |a|\leq2W^{1/(100\theta^2)}, \quad
 W^{1-1/(32\theta^2)}\leq D\leq W^{\theta-1/2}.
 \end{gathered}
\end{equation}
Here $c_1>0$ is an absolute constant to be chosen.
The assumption on $\theta$ will ensure that the integer determinants vanish and that the square sieve condition in Corollary~\ref{C:anisotropic-square} holds.
We now state the main result of this section, which we will prove at the end.
\begin{proposition}\label{L:uniform-unbalanced-count}
Let $a\ne0,-1$ be squarefree and assume~\eqref{e:unbalanced-parameter-assumptions} with sufficiently small $c_1$.
Then
\[
 N_a(D,W)\ll\theta^3\log^{27}(2W)|a|^{7/8}W^{1-1/(32\theta^2)}.
\]
The implied constant is absolute.
\end{proposition}

\subsection{Auxiliary forms}

We keep the hypotheses of Proposition~\ref{L:uniform-unbalanced-count}.
We seek an auxiliary equation for the solutions in each pair of short intervals in the embedding parameters.
Choose
\begin{equation}\label{e:unbalanced-localization}
 p=1-\frac1{8\theta^2},\qquad
 M=\lceil W^p\rceil,\qquad
 B=\lceil\theta^2\rceil.
\end{equation}
For solutions counted by $N_a(D,W)$ in~\eqref{e:unbalanced-count}, the equation $d^2w=n^2+a$ and the bounds~\eqref{e:unbalanced-parameter-assumptions} give
\[
 \frac{DW^{1/2}}{\sqrt2}\leq n<3DW^{1/2},\qquad
 n\geq\max(2,M)|a|^{1/2},\qquad M|a|^{1/2}\ll D,
\]
provided that $c_1$ is sufficiently small.
We may therefore use Lemmas~\ref{L:square-map-coordinates} and~\ref{L:square-map-uniformity} to represent each solution by primitive $(x_1,x_2)$ and $(y_1,y_2)$.
The corresponding elements $\alpha,\beta\in\mathcal O_K$ have absolute norms $E_1,E_2$ from~\eqref{e:invariant-scalar-data} and embedding ratios $\rho_1,\rho_2$ from~\eqref{e:embedding-ratios}.

Fix $k,r,\ell,z_2$ and one of the $O(\log(2|a|))$ sets in Lemma~\ref{L:embedding-coordinates}.
Use its parameters $t_1,t_2$, defined by~\eqref{e:normalized-embedding-coordinates} when $K$ is real and by~\eqref{e:angular-embedding-coordinates} when $K$ is imaginary.
Within this set, let $J_1,J_2$ be intervals for $t_1,t_2$ chosen as in Lemma~\ref{L:embedding-intervals}.
We count solutions with the fixed parameters above that satisfy
\begin{equation}\label{e:unbalanced-interval-pair}
 t_1\in J_1,\qquad t_2\in J_2,\qquad |J_1|=|J_2|=M^{-1}.
\end{equation}
As shown in the proof of Lemma~\ref{L:embedding-intervals}, each interval occurs in only $O(1)$ pairs containing solutions, and there are $O(M)$ such pairs.

With the parameters above, an auxiliary form for $(J_1,J_2)$ is a nonzero integral bihomogeneous form $\Psi$ of bidegree at most $(1,B)$ such that
\[
 \Psi(x_1,x_2;y_1,y_2)=0
\]
at the coordinate pairs of all solutions in those intervals.
Recall $H=2(B+1)$ is the number of monomials of bidegree $(1,B)$.
\begin{lemma}\label{L:unbalanced-auxiliary}
Every pair $(J_1,J_2)$ containing at least $H$ solutions admits such an auxiliary form.
\end{lemma}

\begin{proof}
Form the integer matrix whose columns are the values of
\[
 x_1^i x_2^{1-i}y_1^{B-j}y_2^j,\qquad i\in\{0,1\},\quad 0\leq j\leq B,
\]
at the corresponding coordinate pairs.
We will show that every minor of order $H$ vanishes.
Let $\Delta_1$ be such a minor, formed from $H$ points, and suppose that $\Delta_1\ne0$.

Apply the coordinate changes
\[
 \begin{aligned}
 (x_1,x_2)&\mapsto(\alpha,\bar\alpha)
   =(x_1+x_2\omega,x_1+x_2\bar\omega),\\
 (y_1,y_2)&\mapsto(\beta,\bar\beta)
   =(y_1+y_2\omega,y_1+y_2\bar\omega).
 \end{aligned}
\]
Each has determinant $\bar\omega-\omega$.
The eigenvalue formula for symmetric powers~\cite[Section~I.5, p.~18]{bhatia} shows that the induced change on degree $B$ monomials in $(y_1,y_2)$ has determinant $(\bar\omega-\omega)^{B(B+1)/2}$.
This change occurs twice, while the linear change in $(x_1,x_2)$ occurs $B+1$ times.
Thus $|\Delta_1|$ is multiplied by
\[
 |\omega-\bar\omega|^{B(B+1)+(B+1)}
 =\Delta^{(B+1)^2/2},\qquad \Delta=|\disc K|.
\]

For the $t$-th point, write $E_{1,t},E_{2,t}$ for its absolute norms.
Choose continuous square roots of $\rho_1,\rho_2$ on the fixed sign range or arcs.
Dividing row $t$ by a factor of absolute value $E_{1,t}^{1/2}E_{2,t}^{B/2}$ leaves entries of the form
\[
 \rho_1^{i-1/2}\rho_2^{j-B/2},\qquad i\in\{0,1\},\quad 0\leq j\leq B.
\]

To apply Lemma~\ref{L:determinant-larger-error}, we express these entries in terms of parameters close to zero.
Let $u$ be the left endpoint of $J_1$.
We make the translation
\[
 (t_1,t_2)\mapsto(t_1-u,t_2-2u)
\]
and use the same notation $t_1,t_2$ for the translated parameters.
This preserves $t_2-2t_1$, so
\begin{equation}\label{e:centred-embedding-coordinates}
 0\leq t_1<M^{-1},\qquad
 |t_2-2t_1|\ll |a|^{1/2}/n.
\end{equation}
By~\eqref{e:normalized-embedding-coordinates} and~\eqref{e:angular-embedding-coordinates}, the ratios $\rho_1,\rho_2$ are fixed nonzero multiples of $e^{t_1},e^{t_2}$ when $K$ is real, and of $e^{it_1},e^{it_2}$ when $K$ is imaginary.
Dividing out these fixed column factors leaves the absolute value of the determinant unchanged, since the exponents of each ratio sum to zero over all columns.

When $K$ is real, take $\alpha_t=t_1$ and $\beta_t=t_2-2t_1$ for the $t$-th point.
When $K$ is imaginary, take $\alpha_t=it_1$ and $\beta_t=i(t_2-2t_1)$.
The resulting matrix $\mathcal E$ has entries
\[
 \mathcal E_{t,(i,j)}
 =\exp\bigl((i+2j-B-\tfrac12)\alpha_t+(j-\tfrac B2)\beta_t\bigr).
\]
Thus it has the form~\eqref{e:general-log-evaluation-matrix}, with $(a,b)=(i,j)$, $\sigma=B+\tfrac12$, and $\tau=B/2$.
The two determinants are related by
\[
 |\Delta_1|=\Delta^{-(B+1)^2/2}
 \left(\prod_{t=1}^H E_{1,t}^{1/2}E_{2,t}^{B/2}\right)|\det\mathcal E|.
\]

To bound the normalized determinant, write
\[
 q=\frac12+\frac{\log D}{\log W},\qquad X=DW^{1/2}=W^q.
\]
The bounds on $D$ give $p<q\leq\theta$.
By~\eqref{e:centred-embedding-coordinates} and $n\asymp X$, we have
\[
 |\alpha_t|\ll M^{-1}\leq W^{-p},\qquad
 |\beta_t|\ll |a|^{1/2}/X.
\]
Also $M|a|^{1/2}\ll D$ gives $|a|^{1/2}<X/W^p$ after decreasing $c_1$ if necessary.
We may therefore apply Lemma~\ref{L:determinant-larger-error} with interval scale $W^p$ and $T=|a|^{1/2}$.
It gives
\[
 \begin{split}
 \log|\det\mathcal E|
 \leq{}&\left(-pB(B+1)-(B+1)q+\frac{q^2}{4p}\right)\log W\\
 &+\frac B2\log |a|+O(B^2).
 \end{split}
\]

We now substitute the norms~\eqref{e:invariant-scalar-data} into the relation between the two determinants.
Recall that $d_2=\gcd(d,2)$ and $d'=d/d_2$.
Since $d'\asymp D=W^{q-1/2}$ and $w\asymp W$ at each point, we obtain
\[
 \begin{aligned}
 \log|\Delta_1|
 ={}&\log|\det\mathcal E|+(B+1)(B+q-\tfrac12)\log W\\
 &+(B+1)\bigl((2B+1)\log k-2\log r-2B\log\ell\bigr)\\
 &-\frac{(B+1)^2}{2}\log\Delta+O(B^2).
 \end{aligned}
\]
The terms involving $k,r,\ell,\Delta$ are at most $\tfrac12B(B+1)\log |a|+O(B^2)$, since $k\ll |a|^{1/2}$, $r,\ell\geq1$, and $\Delta\asymp |a|$.

Combining this with the estimate for $\det\mathcal E$, we obtain
\[
 \log|\Delta_1|
 \leq-\Phi\log W+\frac{B(B+2)}2\log |a|+O(B^2\log(2B)),
\]
where, as in~\eqref{e:determinant-gain},
\[
 \Phi=\frac{B+1}{2}-\frac{q^2}{4p}-(1-p)B(B+1).
\]

It remains to show that this upper bound is negative.
The bound on $a$ gives
\[
 \log |a|\leq\frac{\log W}{100\theta^2}+O(1).
\]
Using $q\leq\theta$ and~\eqref{e:unbalanced-localization}, we find
\[
 \Phi-\frac{B(B+2)}{200\theta^2}
 \geq\frac{B+1}{2}-\frac{\theta^2}{4p}
      -\frac{B(B+1)}{8\theta^2}
      -\frac{B(B+2)}{200\theta^2}
 \geq\frac{\theta^2}{16}.
\]
The last inequality uses $\theta\geq3$.
Thus
\[
 \log|\Delta_1|
 \leq-\frac{\theta^2}{16}\log W+O(\theta^4\log(2\theta)).
\]
The upper bound on $\theta$ in~\eqref{e:unbalanced-parameter-assumptions} makes the right hand side negative when $c_1$ is sufficiently small.
This contradicts $|\Delta_1|\geq1$ for a nonzero integer determinant.
Hence every minor of order $H$ vanishes, and the evaluation matrix has rank less than $H$.
A nontrivial rational relation among its columns, after clearing denominators, gives the required form $\Psi$.
\end{proof}

We next bound the coefficients of $\Psi$ and the values of $K_0$ needed for the square sieve.
Fix a pair of intervals $J_1,J_2$ containing solutions of~\eqref{e:unbalanced-interval-pair}.
Let $\Lambda_1,\Lambda_2$ be the lattices for $\alpha,\beta$ from Lemma~\ref{L:embedding-lattices}, and put
\[
 R_1=\frac{Dk}{r^2},\qquad R_2=W\left(\frac{k}{\ell}\right)^2.
\]
The norm identities~\eqref{e:invariant-scalar-data} give $E_i\asymp R_i$.

For $i=1,2$, write $(x,y)\in\Lambda_i$ for the embedding coordinates of $\alpha$ or $\beta$, respectively.
If $t_{i,0}$ is the left endpoint of $J_i$, put $s_{i,0}=s_i(t_{i,0})$, where $s_i(t_i)=x/y$ is the slope function from Lemma~\ref{L:embedding-lattices}.
For a solution with $t_i\in J_i$, the lemma gives
\[
 |y|\asymp R_i^{1/2},\qquad
 |x-s_{i,0}y|\ll\frac{R_i^{1/2}}M.
\]
Apply~\eqref{e:intrinsic-interval-map} with $R=R_i$ and $s_0=s_{i,0}$, and denote the image lattice by $\Lambda_{J_i}$.
For $i=2$, this is the lattice used in Lemma~\ref{L:embedding-intervals}.
The two coordinates of each transformed solution vector are bounded by absolute constants.

Choose reduced bases of $\Lambda_{J_1}$ and $\Lambda_{J_2}$, and write $(x_1,x_2)$ and $(y_1,y_2)$ for the integer coordinates in these bases, respectively.
Equation~\eqref{e:square-map-integral-equation} becomes
\[
 G_0(x_1,x_2;y_1,y_2)=\xi,\qquad
 \xi=\frac{k^2z_2}{r^2\ell}\in\ZZ\setminus\{0\},
\]
where $G_0$ has bidegree $(2,1)$.

For an auxiliary form $\Psi$ that depends on $(x_1,x_2)$, write $\Psi=D_0F$ and define $K_0$ as in Lemma~\ref{L:linear-elimination}, with $\xi$ in place of $h$.
By~\eqref{e:reduced-basis-bounds}, every solution with $t_1\in J_1$ and $t_2\in J_2$ satisfies
\[
 |y_1|,|y_2|\leq T,\qquad T=\frac{C}{\lambda_1(\Lambda_{J_2})}\geq1,
\]
for a sufficiently large absolute constant $C$.
We use this value of $T$ in the definition~\eqref{e:binary-form-value-size} of $W(K_0)$.

\begin{lemma}\label{L:unbalanced-value-bound}
The auxiliary form $\Psi$ in Lemma~\ref{L:unbalanced-auxiliary} can be chosen so that its coefficients in the reduced coordinates satisfy
\[
 \log\lVert\Psi\rVert\ll\theta^4\log(2|a|W).
\]
For this choice of $\Psi$, if $\Psi$ depends on $(x_1,x_2)$ and the system
\[
 G_0(x_1,x_2;y_1,y_2)=\xi,\qquad F(x_1,x_2;y_1,y_2)=0
\]
has a solution with both $(x_1,x_2)$ and $(y_1,y_2)$ primitive, then
\[
 \log W(K_0)\ll\theta^4\log(2|a|W).
\]
The implied constants are absolute.
\end{lemma}

\begin{proof}
Let $S\geq1$ bound the absolute values of the reduced coordinates of $\alpha$.
By~\eqref{e:invariant-scalar-data}, the norms satisfy
\[
 E_1\ll |a|^{1/2}W^\theta,
 \qquad
 E_2\ll |a|W.
\]
Lemma~\ref{L:embedding-lattices} bounds the changes from integral coordinates to the rescaled or rotated embedding coordinates, and their inverses, by $O(|a|)$.
Using these bounds and $M\ll W$ in~\eqref{e:intrinsic-interval-map}, estimate~\eqref{e:reduced-change-bounds} bounds the changes to reduced bases and their inverses by $(2|a|W)^{O(\theta)}$ for $\alpha$ and $(2|a|W)^{O(1)}$ for $\beta$, and gives
\[
 \log(2S)\ll \theta\log(2|a|W),
 \qquad
 \log(2T)\ll\log(2|a|W).
\]

The unimodular coordinate changes preserve the rank of the evaluation matrix, which is less than $H$ by Lemma~\ref{L:unbalanced-auxiliary}.
Choose $\Psi$ by Lemma~\ref{L:quantitative-two-layer-height}.
Together with~\eqref{e:quantitative-factor-height}, this gives
\[
 \log\lVert\Psi\rVert,\ \log\lVert F\rVert
 \ll\theta^4\log(2|a|W)
\]
for every primitive factor $F$ of $\Psi$.
Now suppose that $\Psi$ depends on $(x_1,x_2)$ and that $G_0=\xi$, $F=0$ have a solution with both coordinate pairs primitive.
The same coordinate bounds give
\[
 \log\lVert G_0\rVert+\log(2|\xi|)
 \ll \theta\log(2|a|W).
\]
Substituting these estimates into~\eqref{e:quantitative-value-height}, with $\xi$ in place of $h$, gives the claimed bound.
\end{proof}

\subsection{Counting on the auxiliary curves}

We now count points on the auxiliary curves, starting with factors that are independent of $(x_1,x_2)$.
Each such factor restricts $(y_1:y_2)$ to finitely many directions.
For each corresponding primitive $(y_1,y_2)$, Lemma~\ref{L:quadratic-unit-count} bounds the number of solutions of the remaining quadratic equation in $(x_1,x_2)$.
To sum these bounds over $k,r,\ell,z_2$, we use the following elementary divisor estimate.
\begin{lemma}\label{L:scalar-divisor-average}
For integers $j\geq1$ and $z_2$ with $0<|z_2|\leq2$, and $T\geq2$, we have
\[
 \sum_{k\leq T}\tau(k^2z_2)^j
 \ll_j T\log^{3^j-1}(2T).
\]
\end{lemma}

\begin{proof}
Let $\tau_r(n)$ count the ordered factorizations of $n$ into $r$ positive integers.
Comparing prime powers gives $\tau(n^2)\leq\tau_3(n)$ and $\tau_r(n)\tau_s(n)\leq\tau_{rs}(n)$.
Since $|z_2|\leq2$, it follows that $\tau(k^2z_2)^j\leq2^j\tau_{3^j}(k)$.
The elementary estimate
\[
 \sum_{k\leq T}\tau_r(k)
 \leq T\left(\sum_{n\leq T}\frac1n\right)^{r-1}
 \ll_r T\log^{r-1}(2T),
\]
with $r=3^j$, proves the lemma.
\end{proof}

We retain the hypotheses of Proposition~\ref{L:uniform-unbalanced-count} and the choices of $M,B$ in~\eqref{e:unbalanced-localization}.
For each pair $(J_1,J_2)$ in~\eqref{e:unbalanced-interval-pair} containing at least $H=2(B+1)$ solutions counted by $N_a(D,W)$ from~\eqref{e:unbalanced-count}, choose an auxiliary form $\Psi$ as in Lemma~\ref{L:unbalanced-value-bound}.
In the reduced coordinates $(x_1,x_2)$ and $(y_1,y_2)$, use the factorization from Lemma~\ref{L:linear-elimination}:
\[
 \Psi(x_1,x_2;y_1,y_2)=D_0(y_1,y_2)F(x_1,x_2;y_1,y_2).
\]
If $\Psi$ is independent of $(x_1,x_2)$, put $D_0=\Psi$ and $F=1$.

We separate the solutions according to whether $D_0(y_1,y_2)=0$ and define
\begin{equation}\label{e:unbalanced-curve-counts}
 \begin{aligned}
 S_1&=\#\{(x_1,x_2,y_1,y_2):D_0(y_1,y_2)=0\},\\
 S_2&=\#\{(x_1,x_2,y_1,y_2):D_0(y_1,y_2)\ne0,\ F(x_1,x_2;y_1,y_2)=0\}.
 \end{aligned}
\end{equation}
Both counts run over all choices of $k,r,\ell,z_2$ in Lemma~\ref{L:square-map-coordinates} and all such interval pairs in every set from Lemma~\ref{L:embedding-coordinates}.
In the next two lemmas, we bound $S_1$ using binary quadratic equations and $S_2$ using the square sieve.
\begin{lemma}\label{L:uniform-exceptional-factors}
Under the hypotheses of Proposition~\ref{L:uniform-unbalanced-count}, we have
\[
 S_1\ll\theta^3|a|^{1/2}M\log^{27}(2|a|).
\]
The implied constant is absolute.
\end{lemma}

\begin{proof}
For fixed $k,r,\ell,z_2$ and one pair of intervals, $D_0$ has degree at most $B$, so $D_0(y_1,y_2)=0$ gives at most $B$ rational slopes $y_1/y_2$ in the reduced coordinates.
We return to the original integral coordinates of Lemma~\ref{L:square-map-coordinates}, again writing them as $(x_1,x_2)$ and $(y_1,y_2)$.
Each ratio in the reduced coordinates fixes the original primitive $(y_1,y_2)$ up to sign, and $\Nm\beta=w(k/\ell)^2$ then determines $w$.
By~\eqref{e:embedding-scale-lower-bound}, $\ell\mid k$, so we may apply Lemma~\ref{L:quadratic-unit-count} with $f=k/\ell$.
Lemma~\ref{L:square-map-uniformity} bounds the integral coordinates by $(2|a|W)^{O(\theta)}$.
Lemma~\ref{L:quadratic-unit-count} therefore bounds the number of points in each primitive direction by
\[
 \ll\tau(\xi)\frac{\theta\log(2|a|W)}{\log(2w)}
 \ll\theta\tau(\xi),
\]
since $w\sim W$ and $\log(2|a|)\ll\log W$.

Summing over the $O(M)$ pairs of intervals in each of the $O(\log(2|a|))$ sets gives
\[
 O(B\theta\tau(\xi)M\log(2|a|))
\]
points with $D_0(y_1,y_2)=0$ for fixed $k,r,\ell,z_2$.
For fixed $k,z_2$, the divisibility condition $r^2\ell\mid k^2z_2$ in Lemma~\ref{L:square-map-coordinates} gives at most $\tau(k^2z_2)^2$ choices of $(r,\ell)$, each with $\tau(\xi)\leq\tau(k^2z_2)$.
Summing over $k\ll |a|^{1/2}$ by Lemma~\ref{L:scalar-divisor-average} with $j=3$ and using $B\ll\theta^2$ gives the claimed bound for $S_1$.
\end{proof}

We now use the square sieve to bound $S_2$.
\begin{lemma}\label{L:unbalanced-residual-count}
Under the hypotheses of Proposition~\ref{L:uniform-unbalanced-count}, we have
\[
 S_2\ll |a|^{7/8}W^{3/4}M^{1/4}\log^{6}(2|a|W).
\]
The implied constant is absolute.
\end{lemma}

\begin{proof}
Fix $k,r,\ell,z_2$ and one set from Lemma~\ref{L:embedding-coordinates}.
Recall the norm scale and discriminant:
\[
 R=W(k/\ell)^2,\qquad \Delta=|\disc K|.
\]
For each interval pair contributing to $S_2$, Lemma~\ref{L:linear-elimination}, with $\xi$ in place of $h$, gives a nonzero binary form $K_0$ of odd degree at most $2B+1$ whose values are nonzero squares, with at most two first coordinate pairs for each second pair.
The unimodular changes preserve primitivity, and Lemma~\ref{L:unbalanced-value-bound} bounds the values of $K_0$ on these coordinates.
For an interval $I$ in $t_2$, use the lattice $\Lambda_I$ defined before Lemma~\ref{L:embedding-intervals}, with the fixed scale $R\asymp E_2$.
If the counting rectangle contains two independent lattice vectors, then $\lambda_2(\Lambda_I)\ll1$.
By~\eqref{e:reduced-basis-bounds}, the reduced coordinates lie in intervals $I_1,I_2\subseteq[-T,T]$, with $T=C/\lambda_1(\Lambda_I)$ as in Lemma~\ref{L:unbalanced-value-bound} and
\[
 |I_i|\asymp\frac1{\lambda_i(\Lambda_I)},\qquad
 1\leq|I_2|\leq|I_1|.
\]
Put
\[
 J=\left(\frac{R}{M\Delta^{1/2}}\right)^{1/4}.
\]
The map~\eqref{e:intrinsic-interval-map} and~\eqref{e:reduced-basis-bounds} give
\[
 \det\Lambda_I\asymp\frac{M\Delta^{1/2}}R,\qquad
 (|I_1||I_2|)^{1/4}\asymp J.
\]
Since $R\geq W$ by~\eqref{e:embedding-scale-lower-bound} and $\Delta\asymp |a|$, we have
\[
 J\gg W^{3/(100\theta^2)}.
\]
The restriction on $\theta$ gives $J\gg(\log W)^4$ after decreasing $c_1$, while Lemma~\ref{L:unbalanced-value-bound} gives $\log W(K_0)\ll(\log W)^3$.
Thus the prime condition in Corollary~\ref{C:anisotropic-square} holds with interval lengths $|I_i|$ and value bound $W(K_0)$.

Applying the square sieve in Corollary~\ref{C:anisotropic-square} to $K_0$, together with the multiplicity bound in Lemma~\ref{L:linear-elimination}, bounds the contribution from each rectangle with two independent vectors by
\[
 \log^2(2|a|W)\frac{J}{\lambda_1(\Lambda_I)}.
\]
The same bound holds for rectangles with one primitive direction, since they contribute at most four solutions by Lemma~\ref{L:linear-elimination} and $J/\lambda_1(\Lambda_I)\gg1$.
By the proof of Lemma~\ref{L:embedding-intervals}, each interval in $t_2$ occurs with only $O(1)$ intervals in $t_1$.
We can therefore sum the preceding bound using Lemma~\ref{L:embedding-intervals}.
The bounds~\eqref{e:unbalanced-parameter-assumptions} and~\eqref{e:unbalanced-localization} give $\Delta\asymp|a|\ll M^2$, so
\[
 M^{-1/4}\Delta^{-1/8}\ll M^{1/4}\Delta^{-3/8}.
\]
Substituting the value of $J$ and summing over the $O(\log(2|a|))$ sets bounds the contribution for fixed $k,r,\ell,z_2$ by
\begin{equation}\label{e:unbalanced-fixed-parameter-count}
 \log^4(2|a|W)R^{3/4}M^{1/4}\Delta^{-3/8}.
\end{equation}
Here we used $R\ll |a|W$ to bound the factor $\log(2R)$ in Lemma~\ref{L:embedding-intervals}.

For fixed $k,z_2$, there are at most $\tau(k^2z_2)$ choices of $r$, and the sum of $\ell^{-3/2}$ over all $\ell\geq1$ is bounded.
Thus
\[
 \sum_{r^2\ell\mid k^2 z_2}\ell^{-3/2}
 \ll\tau(k^2z_2).
\]
Using $k\ll |a|^{1/2}$ and Lemma~\ref{L:scalar-divisor-average} with $j=1$ gives
\[
 \sum_{k\ll |a|^{1/2}}
 \sum_{r^2\ell\mid k^2 z_2}(k/\ell)^{3/2}
 \ll |a|^{5/4}\log^2(2|a|).
\]
Since $R^{3/4}=W^{3/4}(k/\ell)^{3/2}$, $\Delta\asymp |a|$, and $0<|z_2|\leq2$, summing~\eqref{e:unbalanced-fixed-parameter-count} gives the claimed bound for $S_2$.
\end{proof}

We now combine these bounds to prove Proposition~\ref{L:uniform-unbalanced-count}.
\begin{proof}[Proof of Proposition~\ref{L:uniform-unbalanced-count}]
We first count the points in interval pairs containing fewer than $H$ points.
For fixed $k,r,\ell,z_2$, the partition in Lemma~\ref{L:embedding-coordinates} and the count in the proof of Lemma~\ref{L:embedding-intervals} give $O(M\log(2|a|))$ interval pairs, so these points contribute $O(HM\log(2|a|))$ in total.
For fixed $k,z_2$, there are at most $\tau(k^2z_2)^2$ pairs $(r,\ell)$ with $r^2\ell\mid k^2z_2$.
Summing over $k\ll |a|^{1/2}$ and $0<|z_2|\leq2$ by Lemma~\ref{L:scalar-divisor-average} with $j=2$ and using $H\ll\theta^2$ bounds the number of points in these interval pairs by
\[
 O\left(\theta^2|a|^{1/2}M\log^{9}(2|a|)\right).
\]

In every interval pair containing at least $H$ points, the auxiliary form $\Psi=D_0F$ vanishes at all the points, so their contribution is counted by $S_1+S_2$ in~\eqref{e:unbalanced-curve-counts}.
Lemmas~\ref{L:uniform-exceptional-factors} and~\ref{L:unbalanced-residual-count} bound $S_1$ and $S_2$, respectively.

Since $M\ll W$ and $|a|\geq1$, we have $|a|^{1/2}M\ll W^{3/4}|a|^{7/8}M^{1/4}$.
Adding these bounds to the contribution from pairs of intervals containing fewer than $H$ points therefore gives
\[
 N_a(D,W)\ll\theta^3|a|^{7/8}W^{3/4}M^{1/4}\log^{27}(2|a|W).
\]
Substituting $M\asymp W^{1-1/(8\theta^2)}$ and using $\log(2|a|W)\ll\log(2W)$, we obtain
\[
 N_a(D,W)\ll\theta^3|a|^{7/8}W^{1-1/(32\theta^2)}\log^{27}(2W).\qedhere
\]
\end{proof}

\section{Proofs of the main theorems}\label{s:main-proofs}

\subsection{Proof of Theorem~\ref{T:uniform-exceptional}}\label{s:uniform-exceptional}

We first combine the estimates of Sections~\ref{s:balanced} and~\ref{s:unbalanced} to count squarefree $w$ with fundamental solutions $\varepsilon_w\leq w^\theta$.

\begin{proposition}\label{P:squarefree-exceptional}
There is an absolute constant $c_0>0$ such that, for all sufficiently large $X$,
\[
 \#\{w\sim X:w\text{ squarefree},\ \varepsilon_w\leq w^\theta\}
 \ll X^{1-2c_0/\theta^2},
\]
uniformly for $3\leq\theta\leq c_0\sqrt{\log X/\log\log X}$.
The implied constant is absolute.
\end{proposition}

\begin{proof}
Choose an absolute constant
\[
 0<c_0<1/800
\]
small enough to meet the restrictions on $\theta$ in Propositions~\ref{P:balanced-ranges} and~\ref{L:uniform-unbalanced-count} at scales $UV\asymp X$ and $W=X$, respectively.
The hypothesis on $\theta$ then gives
\[
 \theta^7(\log X)^{28}\ll X^{c_0/\theta^2}.
\]

For squarefree $w\sim X$ with $\varepsilon_w\leq w^\theta$, use the factorization~\eqref{e:factor-equation}.
In the balanced range $u,v>X^{1/(100\theta^2)}$, each box in~\eqref{e:scales} has $X/4<UV<2X$, $U,V>X^{1/(100\theta^2)}/2\geq (UV)^{1/(100\theta^2)}/4$, and $Y\leq\varepsilon_w<(4UV)^\theta$.
Summing the bound of Proposition~\ref{P:balanced-ranges} over the $O((\theta\log X)^4)$ boxes gives a contribution
\[
 \ll\theta^7(\log X)^7X^{1-1/(200\theta^2)}.
\]

In the unbalanced range, assume $v\leq X^{1/(100\theta^2)}$ by symmetry and put $a=-hv$ as in~\eqref{e:unbalanced-equation}.
As in Section~\ref{s:unbalanced}, $a$ is squarefree and $a\ne0,-1$, while $|a|\leq2X^{1/(100\theta^2)}$.
Since $d^2u\leq\varepsilon_w\leq(2X)^\theta$ and $uv\sim X$, we have
\[
 d\leq\sqrt2(2X)^{(\theta-1)/2}X^{1/(200\theta^2)}\leq X^{\theta-1/2}.
\]
We therefore sum $N_a(D,X)$, defined in~\eqref{e:unbalanced-count} with $W=X$, over these values of $a$ and the $O(\theta\log X)$ dyadic intervals with $D\leq X^{\theta-1/2}$.
For $D\leq X^{1-1/(32\theta^2)}$, Lemma~\ref{L:direct-count} gives a total contribution
\[
 \ll(\log X)^2X^{1-17/(800\theta^2)}.
\]
For larger $D$, Proposition~\ref{L:uniform-unbalanced-count} gives a total contribution
\[
 \ll\theta^4(\log X)^{28}X^{1-1/(80\theta^2)}.
\]
Since $\theta\geq3$, combining these three contributions gives a total of
\[
 \ll\theta^7(\log X)^{28}X^{1-1/(200\theta^2)}\ll X^{1-2c_0/\theta^2},
\]
since $c_0<1/800$.
\end{proof}

We pass from squarefree to general nonsquare integers by writing $m=s^2w$, with $w>1$ squarefree, and using
\[
 \varepsilon_w\leq\varepsilon_m.
\]
Indeed, writing $\varepsilon_m=t_m+u_m\sqrt m$, we have $t_m^2-w(su_m)^2=1$.
We now prove Theorem~\ref{T:uniform-exceptional} from the squarefree bound in Proposition~\ref{P:squarefree-exceptional}.

\begin{proof}[Proof of Theorem~\ref{T:uniform-exceptional}]
Choose $c_0$ as in Proposition~\ref{P:squarefree-exceptional} and put $c=c_0/4$.
We prove the theorem in the range
\[
 3\leq\theta\leq c\sqrt{\frac{\log X}{\log\log X}}.
\]
Write each nonsquare integer $m<X$ uniquely as $m=s^2w$, with $s\geq1$ and $w>1$ squarefree.
For $w<X^{1/2}$, the number of such integers is at most
\[
 \sum_{1<w<X^{1/2}}\left(\frac Xw\right)^{1/2}\ll X^{3/4}.
\]
For $w\geq X^{1/2}$, the inequality $\varepsilon_w\leq\varepsilon_m$ gives
\[
 \varepsilon_m\leq m^\theta
 \quad\Rightarrow\quad
 \varepsilon_w\leq\varepsilon_m\leq X^\theta\leq w^{2\theta}.
\]
Cover this range by intervals $w\sim W$ with $W=2^jX^{1/2}<X$, where $j\geq0$ is an integer.
Proposition~\ref{P:squarefree-exceptional} applies with parameter $2\theta$ on each interval, since
\[
 2\theta\leq\frac{c_0}{2}\sqrt{\frac{\log X}{\log\log X}}
 \leq c_0\sqrt{\frac{\log W}{\log\log W}}.
\]
Each $w$ gives at most $(X/W)^{1/2}$ choices of $s$, so the contribution from one interval is
\[
 \ll\left(\frac XW\right)^{1/2}W^{1-c_0/(2\theta^2)}
 =X^{1/2}W^{1/2-c_0/(2\theta^2)}.
\]
Since $1/2-c_0/(2\theta^2)\geq1/4$, the sum over these dyadic intervals is bounded by a constant multiple of its largest term.
Therefore
\[
 E_\theta(X)\ll X^{3/4}+X^{1-c_0/(2\theta^2)}\ll X^{1-c/\theta^2}.
\]
This completes the proof of Theorem~\ref{T:uniform-exceptional}.
\end{proof}

\subsection{Proof of Theorem~\ref{T:field-class-average}}\label{s:class-number-proof}

Recall that the set $\mathcal D$ consists of the positive nonsquare integers congruent to $0$ or $1$ modulo $4$.
We compare the units defined in~\eqref{e:totally-positive-unit} and~\eqref{e:pell-fundamental-unit}.

\begin{lemma}\label{L:unit-comparison}
For every $D\in\mathcal D$,
\[
 \eta_D^+\leq\varepsilon_D\leq(\eta_D^+)^3,
 \qquad
 h^+(D)\leq\frac{3\sqrt D L(1,\chi_D)}{\log\varepsilon_D}.
\]
\end{lemma}

\begin{proof}
Recall that $\mathcal O_D=\ZZ[(D+\sqrt D)/2]$.
Since $\ZZ[\sqrt D]=\ZZ+2\mathcal O_D$ and $|(\mathcal O_D/2\mathcal O_D)^\times|\leq3$, we have
\[
 (\eta_D^+)^j\in1+2\mathcal O_D\subseteq\ZZ[\sqrt D],
 \qquad \text{for some }1\leq j\leq3.
\]
The minimality of $\varepsilon_D$ gives $\varepsilon_D\leq(\eta_D^+)^3$, while that of $\eta_D^+$ gives $\eta_D^+\leq\varepsilon_D$.
The class number bound follows from~\eqref{e:ordinary-class-formula}.
\end{proof}

We next extend a second moment estimate from fundamental discriminants to all order discriminants.
\begin{lemma}\label{L:order-second-moment}
For $X\geq2$, we have
\[
 \sum_{\substack{D\in\mathcal D\\D<X}}L(1,\chi_D)^2\ll X.
\]
\end{lemma}

\begin{proof}
The moment estimate of Granville and Soundararajan~\cite[Theorem~2 and the following remark]{gs} with $z=2$ gives
\begin{equation}\label{e:second-moment}
 \sum_{\substack{1<D_0<X\\D_0\ \text{fundamental}}}L(1,\chi_{D_0})^2\ll X.
\end{equation}
Write each $D\in\mathcal D$ uniquely as $D=D_0f^2$, with $f\geq1$ and $D_0$ a positive fundamental discriminant.
Euler factors give
\[
 L(1,\chi_{D_0f^2})=L(1,\chi_{D_0})\prod_{p\mid f}\left(1-\frac{\chi_{D_0}(p)}p\right).
\]
To bound the Euler factors, use
\[
 \prod_{p\mid f}\left(1+\frac1p\right)
 \leq\sum_{d\mid f}\frac1d\ll\log(2f).
\]
For fixed $f$, we have $D_0<X/f^2$, so~\eqref{e:second-moment} gives
\[
 \sum_{\substack{D\in\mathcal D\\D<X}}L(1,\chi_D)^2
 \ll X\sum_{f\geq1}\frac{\log^2(2f)}{f^2}\ll X.
\]
This proves the lemma.
\end{proof}

\begin{proof}[Proof of Theorem~\ref{T:field-class-average}]
Let $\theta\geq3$ be a parameter, to be chosen later.
At most $E_\theta(X)$ discriminants $D<X$ satisfy $\varepsilon_D\leq D^\theta$.
Cauchy--Schwarz and Lemma~\ref{L:order-second-moment} therefore give
\begin{equation}\label{e:class-average-moment-bounds}
 \sum_{\substack{D\in\mathcal D\\D<X}}L(1,\chi_D)\ll X,
 \qquad
 \sum_{\substack{D\in\mathcal D,\, D<X\\\varepsilon_D\leq D^\theta}}L(1,\chi_D)
 \ll\sqrt{X E_\theta(X)}.
\end{equation}

Split the class number sum according to whether $\varepsilon_D\leq D^\theta$ or $\varepsilon_D>D^\theta$.
We use $\log\varepsilon_D>\tfrac12\log D$ in the first case and $\log\varepsilon_D>\theta\log D$ in the second.
Since $\sqrt D/\log D\ll\sqrt X/\log X$ uniformly for $5\leq D<X$, Lemma~\ref{L:unit-comparison} gives
\begin{equation}\label{e:class-average-cutoff}
 \begin{split}
 \sum_{\substack{D\in\mathcal D\\D<X}}h^+(D)
 &\ll\frac{\sqrt X}{\log X}
 \left(
 \sum_{\substack{D\in\mathcal D, \,  D<X\\\varepsilon_D\leq D^\theta}}L(1,\chi_D)
 +\frac1\theta\sum_{\substack{D\in\mathcal D\\D<X}}L(1,\chi_D)
 \right)\\
 &\ll\frac{X^{3/2}}{\log X}
 \left(\left(\frac{E_\theta(X)}X\right)^{1/2}+\frac1\theta\right).
 \end{split}
\end{equation}
Here the second inequality follows from~\eqref{e:class-average-moment-bounds}.

Let $c$ be the constant in Theorem~\ref{T:uniform-exceptional}, and choose
\begin{equation}\label{e:class-average-parameter}
 \theta=c\sqrt{\frac{\log X}{\log\log X}}.
\end{equation}
For sufficiently large $X$, this choice satisfies the hypotheses of Theorem~\ref{T:uniform-exceptional}, which gives
\begin{equation}\label{e:class-exceptional-decay}
 \left(\frac{E_\theta(X)}X\right)^{1/2}
 \ll X^{-c/(2\theta^2)}
 =(\log X)^{-1/(2c)}
 \ll \theta^{-1}.
\end{equation}
Substituting into~\eqref{e:class-average-cutoff}, we obtain
\[
 \sum_{\substack{D\in\mathcal D\\D<X}}h^+(D)
 \ll\frac{X^{3/2}}{\theta\log X}
 \ll\frac{X^{3/2}\sqrt{\log\log X}}{(\log X)^{3/2}}.
\]
By~\eqref{e:field-form-average-comparison}, the same bound holds for $S(X)$.
\end{proof}

\subsection{Proof of Theorem~\ref{T:odd-class-average}}\label{s:odd-class-proof}

\begin{proof}
For a real quadratic field $K$ with discriminant $D$, the number $h_{\mathrm{odd}}(K)$ of ideal classes of odd order is the common odd part of $h(K)$ and $h^+(D)$.
Genus theory gives $2^{\omega(D)-1}\mid h^+(D)$, where $\omega(D)$ is the number of distinct prime divisors of $D$.
Lemma~\ref{L:unit-comparison} therefore yields
\[
 h_{\mathrm{odd}}(K)
 \leq\frac{h^+(D)}{2^{\omega(D)-1}}
 \leq\frac{6\sqrt D L(1,\chi_D)}{2^{\omega(D)}\log\varepsilon_D}.
\]
Browning and Chan~\cite[p.~1679, last display]{browning-chan} used the weighted estimate
\[
 \sum_{\substack{1<D<X\\D\ \mathrm{fundamental}}}
 \frac{\sqrt D L(1,\chi_D)}{2^{\omega(D)}\log D}
 \ll\frac{X^{3/2}}{(\log X)^{3/2}}.
\]

For $\theta\geq3$, split the sum as in~\eqref{e:class-average-cutoff}.
For $\varepsilon_D>D^\theta$, use this weighted estimate and the bound $\log\varepsilon_D>\theta\log D$.
For $\varepsilon_D\leq D^\theta$, use $h_{\mathrm{odd}}(K)\leq h^+(D)$ and~\eqref{e:class-average-moment-bounds}.
This gives
\[
 \sum_{0<D_K<X}h_{\mathrm{odd}}(K)
 \ll\frac{X^{3/2}}{\log X}
 \left(\left(\frac{E_\theta(X)}X\right)^{1/2}
           +\frac{1}{\theta\sqrt{\log X}}\right).
\]

We choose $\theta$ as in~\eqref{e:class-average-parameter}, with the same constant $c$.
Since $c\leq1/2$, equation~\eqref{e:class-exceptional-decay} gives $(E_\theta(X)/X)^{1/2}\ll1/\log X$, so the preceding bound gives
\[
 \sum_{0<D_K<X}h_{\mathrm{odd}}(K)
 \ll\frac{X^{3/2}\sqrt{\log\log X}}{(\log X)^2}.
\]
This completes the proof of Theorem~\ref{T:odd-class-average}.
\end{proof}

\end{document}